\documentclass[11pt]{article}

\usepackage[a4paper,left=2cm,right=2cm,top=2.5cm,bottom=2.5cm]{geometry}

\usepackage{float}

\usepackage{amsmath,amssymb,amsthm}

\usepackage{hyperref}

\usepackage{color}
\newcommand{\dee}{d}
\newcommand{\bint}{\displaystyle\int}

\usepackage[utf8]{inputenc}
\usepackage{tikz}
\usetikzlibrary{matrix}
\usepackage[all]{xy}
\usepackage{tabularx}
\newcommand\pder[2][]{\ensuremath{\tfrac{\partial#1}{\partial#2}}}

\newtheorem{Lemma}{Lemma}
\newtheorem{Remark}{Remark}
\usepackage{booktabs}

\title{Explicit Discrete Solution for Some Optimization Problems and Estimations with Respect to the Exact Solution}

\author{
Julieta Bollati$^{1,2}$, Mariela Olguín$^{3}$, Domingo A. Tarzia$^{1,2}$ \\
\small {{$^1$} Depto. Matem\'atica, FCE, Univ. Austral, Paraguay 1950} \\
\small {{$^2$} CONICET, Argentina} \\
\small {{$^3$} Departamento de
Matem\'atica, EFB-FCEIA, Univ. Nacional de Rosario} \\
\small {S2000FZF Rosario, Argentina.}\\
}

\date{}

\begin{document}
\maketitle
\pagestyle{empty}
\abstract{We consider two steady-state heat conduction systems called, $S$ and $S_\alpha$, in a multidimensional bounded domain $D$  for the Poisson equation with source energy $g$.
In one system, we impose mixed boundary conditions (temperature $b$ on the boundary $\Gamma_1$, heat flux $q$ on $\Gamma_2$ and an adiabatic condition on $\Gamma_3$). In the other system,  the condition on $\Gamma_1$ is replaced by a  convective heat flux condition with coefficient $\alpha$. 
For each of these systems, we consider three associated optimization problems $(P_{i})$ and $(P_{i\alpha })$, $i=1,2,3$, where the variable is the source energy $g$, the heat flux $q$ and the environmental temperature $b$, respectively. 
In the particular case where $D$ is a rectangle, the explicit continuous optimization variables and the corresponding state of the systems are known. 
In the present work, by using a finite difference scheme, we obtain the discrete systems $({S^h})$ and ${(S^h_\alpha)}$ and discrete optimization problems ${(P^h_i)}$ and ${(P^h_{i \alpha})}$, $i=1,2,3$, where $h$ is the space step in the discretization. Explicit discrete solutions are found, and  convergence and estimation errors results are proved when $h$ goes to zero and when $\alpha$ goes to infinity. 
Moreover, some numerical simulations are provided in order to test theoretical results. Finally, we note that the use of a three-point finite-difference approximation for the Neumann or Robin boundary condition at the boundary improves the global order of convergence from $O(h)$ to $O(h^2)$.
}

\smallskip
\noindent {\it Keywords:}  Optimal control, finite difference, Explicit solutions. 

\smallskip
\noindent {\it 2000 AMS Subject Classification}{35C05; 49J20; 49K20; 49M25; 65N15; 65N30 
}
\section{Introduction}\label{intro}

   We consider a multidimensional bounded domain $\Omega \subset \mathbb{R}^n$ whose regular boundary $\Gamma$  consists of three disjoint portions $\Gamma_i$ with $meas(\Gamma_i) > 0,$ for   $\; i=1, 2, 3$. We define two~stationary heat conduction problems  $(S)$ and $(S_\alpha)$ with mixed boundary conditions which are given by  \eqref{ec 1.1} and \eqref{ec 1.2}, and by \eqref{ec 1.1} and \eqref{ec 1.3}, respectively: 
 \begin{eqnarray}
 & -\Delta u=g   \quad in \quad  \Omega, &  \label{ec 1.1}\\[0.25cm]
 u\Big\vert_{\Gamma_1}=b,\qquad &-\tfrac{\partial u}{\partial n}\Big\vert_{\Gamma_2} =q , \qquad &-\tfrac{\partial u}{\partial n} \Big\vert_{\Gamma_3}=0, \label{ec 1.2}\\[0.25cm]
  -\tfrac{\partial u}{\partial n} \Big\vert_{\Gamma_1}= \alpha\,(u-b),\qquad & -\tfrac{\partial u}{\partial n} \Big\vert_{\Gamma_2}=q \qquad & -\tfrac{\partial u}{\partial n}\Big\vert_{\Gamma_3} =0 ,\label{ec 1.3}
\end{eqnarray}   
 
\noindent where $g$ is the internal energy of the system in $\Omega$, $b>0$ the environmental temperature  on $\Gamma_1$, $q$ is the heat flux on $\Gamma_2$ and, $\alpha>0$ is the convective heat coefficient on $\Gamma_1.$ We assume that $g \in H= L^2(\Omega)$, $q\in Q=L^2(\Gamma_2)$ and $b \in B= H^{1/2}(\Gamma_1).$  These problems correspond to stationary Stefan problems ~\cite{TaTa1989,Ta1996}.
   Notice that mixed boundary conditions play an important role in several applications, e.g., heat conduction and electric potential problems~\cite{HaMe2009}.

The variational formulation of the  elliptic problems {$(S)$ and $(S_\alpha)$, corresponding to} \eqref{ec 1.1}, \eqref{ec 1.2} and \eqref{ec 1.1}, \eqref{ec 1.3}, {respectively,} can be found  in~\cite{GaTa2003,GaTa2008,TaTa1989}. In general,  solutions of mixed elliptic boundary value problems {are not very} regular~\cite{Gr1985}, but there are cases in which they are regular~\cite{AzKr1982,LaCaBr2008,Sh1968}. Other theoretical optimization problems on the subject have been {studied} in~\cite{Hi2005,WaWa2011}. 
 
We define the optimization problems $(P_i)$ and $(P_{i\,\alpha})$ $i=1,2,3$, associated to the systems $(S)$ and $(S_\alpha)$, respectively (see~\cite{Ba1984,GaTa2003,Li1968,NeSpTi2006,Tr2010}). 

The distributed optimization problems $(P_1)$ and $(P_{1\alpha})$ on the constant internal energy $g$ are formulated as:
\begingroup
\addtolength{\jot}{0.5em}
\begin{align}
& \text{ find }\quad g_{op}\in \mathbb{R}\quad\text{ such that }\quad
J_{1}(g_{op})=\min\limits_{g\in \mathbb{R}} \text{ }J_{1}(g)
\label{PControlJ1} \\
& \text{find }\quad g_{{\alpha}_{op}}\in \mathbb{R}\quad\text{ such that
}\quad J_{1\alpha}(g_{{\alpha}_{op}})=\min\limits_{g\in \mathbb{R}} \text{
}J_{1\alpha}(g)  \label{PControlalfaJ1}
\end{align}
\endgroup
where $J_{1}:\mathbb{R}{\rightarrow}{\Bbb R}_{0}^{+}$ and
$J_{1\alpha}:\mathbb{R}{\rightarrow}{\Bbb R}_{0}^{+}$ are given by
\begin{equation} \label{J1-J1alpha}
J_{1}(g)=\frac{1}{2}\left\|u_{g}-z_{d}\right\|
_{H}^{2}+\frac{M_{1}}{2}\left\| g\right\|_{H}^{2},\qquad J_{1\alpha}(g)=\frac{1}{2}\left\|u_{\alpha
g}-z_{d}\right\| _{H}^{2}+\frac{M_{1}}{2}\left\| g\right\|_{H}^{2}
\end{equation}
with $M_{1}\in\mathbb{R}^+$ and $z_{d}\in \mathbb{R}$. { For each $g \in \mathbb{R}$, $u_g$ and $u_{\alpha g}$ denote the unique solutions to the systems $(S)$ and $(S_\alpha)$, respectively, for given data $q \in \mathbb{R}$ and $b \in \mathbb{R}$.  }
{ Here and throughout this section, $\|\cdot\|_H$ denotes the standard $L^2(\Omega)$ norm.}

The boundary optimization problems $(P_2)$ and $(P_{2\alpha})$ on the constant heat flux $q$ on $\Gamma_{2}$ are defined as:
\begingroup
\addtolength{\jot}{0.5em}
\begin{align}
& \text{find }\quad q_{op}\in \mathbb{R}\quad\text{ such that }\quad
J_{2}(q_{op})=\min\limits_{q\in \mathbb{R}}\,J_{2}(q) \label{PControlJ2} \\
& \text{find }\quad q_{{\alpha}_{op}}\in \mathbb{R}\quad\text{ such that
}\quad J_{2\alpha}(q_{{\alpha}_{op}})=\min\limits_{q\in\mathbb{R}}\,J_{2\alpha}(q)  \label{PControlalfaJ2}
\end{align}
\endgroup
where 
$J_{2}:\mathbb{R}{\rightarrow}{\Bbb R}_{0}^{+}$ and
$J_{2\alpha}:\mathbb{R}{\rightarrow}{\Bbb R}_{0}^{+}$ are given by
\begin{equation}\label{J2-J2alpha}
J_{2}(q)=\frac{1}{2}\left\| u_{q}-z_{d}\right\|
_{H}^{2}+\frac{M_{2}}{2}\left\| q\right\|_{Q}^{2}, \qquad
J_{2\alpha}(q)=\frac{1}{2}\left\| u_{\alpha
q}-z_{d}\right\| _{H}^{2}+\frac{M_{2}}{2}\left\| q\right\|_{Q}^{2}
\end{equation}
with $M_2\in\mathbb{R}^+$ and $z_d\in
\mathbb{R}$. For each $q\in\mathbb{R}$,  we denote with $u_{q}$ and $u_{\alpha q}$  the unique solutions to the
systems $(S)$ and $(S_\alpha)$ respectively,  for data $g\in \mathbb{R}$ and $b\in \mathbb{R}$. 
{ Here and throughout this section, $\|\cdot\|_Q$ denotes the standard $L^2(\Gamma_2)$ norm.}

The boundary optimization problems $(P_3)$ and $(P_{3\alpha})$ on the constant temperature $b$ in an external neighborhood of $\Gamma_{1}$ are set as
\begingroup
\addtolength{\jot}{0.5em}
\begin{align}
& \text{find }\quad b_{op}\in \mathbb{R} \quad\text{ such
that }\quad J_{3}(b_{op})=\min\limits_{b\in
\mathbb{R}}\,J_{3}(b)
\label{PControlJ3} \\
& \text{find }\quad b_{{\alpha}_{op}}\in
\mathbb{R}\quad\text{ such that }\quad
J_{3\alpha}(b_{{\alpha}_{op}})=\min\limits_{b\in
\mathbb{R}}\,J_{3\alpha}({b})  \label{PControlalfaJ3}
\end{align}
\endgroup
where $J_{3}:\mathbb{R}{\rightarrow}{\Bbb R}_{0}^{+}$ and
$J_{3\alpha}:\mathbb{R}{\rightarrow}{\Bbb R}_{0}^{+}$, given by
\begin{equation}\label{J3-J3alpha}
J_{3}(b)=\frac{1}{2}\left\| u_{b}-z_{d}\right\|
_{H}^{2}+\frac{M_{3}}{2}\left\| b\right\|_{B}^{2},\qquad J_{3\alpha}(b)=\frac{1}{2}\left\| u_{\alpha b}-z_{d}\right\|
_{H}^{2}+\frac{M_{3}}{2}\left\| b\right\|_{B}^{2}
\end{equation}
\noindent with $M_3\in\mathbb{R}^+$ and $z_d\in \mathbb{R}$. For every $b\in\mathbb{R}$, the functions $u_{b}$ and $u_{\alpha b}$ are the unique solutions of  systems $(S)$ and $(S_\alpha)$ respectively, for data $g\in \mathbb{R}$ and $q\in\mathbb{R}$. 
{Here and throughout this section, $\|\cdot\|_B$ denotes the standard norm in $B=H^{1/2}(\Gamma_1)$.
}

{ In~\cite{BoGaTa2020}, explicit solutions to the continuous systems $(S)$ and $(S_\alpha)$ were derived, together with  the associated optimization problems $(P_i)$ and $(P_{i\alpha})$ for $i=1,2,3$, in the particular case where the domain is a rectangle. These explicit solutions serve as a rigorous benchmark for assessing the accuracy and reliability of numerical methods.}
  
{The aim of this paper is three-fold: 
(i) to obtain explicit solutions to the systems $(S)$ and $(S_\alpha)$ in a rectangular domain; 
(ii) to derive explicit discrete solutions for the optimization problems $(P_i)$ and $(P_{i\alpha})$, $i=1,2,3$, using finite difference methods; 
and (iii) to estimate the order of convergence of the discrete solutions by comparison with the exact explicit ones.}
 
It is worth mentioning that there are several articles available in the literature that obtain explicit  discrete  solutions of some optimization problems~\cite{An1968,ArMoUg2011}. 
For example, in~\cite{LaSc2023},  exact formulas are derived for the solution of an optimal boundary control problem governed by the one-dimensional heat equation where the control function  measures the distance of the final state from the target. In~\cite{YaSu2023} a finite element approximation is applied for some kind of parabolic optimal control problems with Neumann boundary conditions. Some numerical experiments are carried out setting a rectangular domain.

This paper is organized as follows: in Section~\ref{sec2} we obtain the discrete explicit solution to the systems $(S)$ and $(S_\alpha)$ by the finite difference method. In Section~\ref{sec3}, we obtain explicit discrete solutions to the discrete distributed optimization problems associated with $(P_1)$ and $(P_{1\alpha})$, respectively, where the  variable is the internal energy $g$. 
In Section~\ref{sec4}, we define  discrete boundary optimization problems where the variable is the heat flux $q$, associated with $(P_2)$ and $(P_{2\alpha})$, respectively, obtaining the discrete explicit  solutions. In the same manner, in Section~\ref{sec5}, we derive explicit discrete solutions to the discrete boundary optimal control problems associated with $(P_3)$ and $(P_{3\alpha})$, respectively, where the optimization variable is $b$.  In all cases, when the step discretization goes to zero, convergence results are obtained by also estimating the order of convergence of the approximate solutions.
In Section~\ref{sec6}, we carry out some numerical simulations  in order to illustrate the theoretical convergence results obtained in the previous sections. Finally,  in Section~\ref{sec7}, we analyze the order of convergence of the discrete systems associated with $(S)$ and $(S_\alpha)$ by considering a modified approximation of the Neumann boundary condition on $\Gamma_2$, which leads to an improved convergence order.

{The explicit continuous solutions of the systems and the associated optimal control problems in a rectangular domain are already available in the literature; in particular, they are given in~\cite{BoGaTa2020}. 

The novelty of the present work can be summarized as follows: 
(i) the derivation of explicit discrete solutions for the state and the control variables; 
(ii) a rigorous analysis of the convergence of the discrete solutions, including the estimation of their orders of convergence; 
and (iii) an improved approximation of the boundary conditions in the discrete framework.}

\section{Discrete Systems for \boldmath{$(S)$} and \boldmath{$(S_\alpha)$}}\label{sec2} 

In this section we obtain the discrete explicit solutions to the  systems $(S)$ and $(S_\alpha)$ in a rectangular  domain in the plane  $\Omega = (0, x_0) \times (0, y_0)$ with $x_0>0$ and $y_0 > 0$. Its boundaries $\Gamma_i$ for $i=1,2,3$  are defined by:

     \[\Gamma_1= \{(0,y):\, y\in (0,y_0]\}, \quad \Gamma_2= \{(x_0,y):\, y\in (0,y_0]\} \]
      and
 \[ \Gamma_3= \{(x,0):\, x\in [0,x_0]\} \cup \{(x,y_0):\, x\in [0,x_0]\}.\]
    
According to~\cite{BoGaTa2020}, the continuous solutions, in $\Omega$, for the systems $(S)$ and $(S_\alpha)$ defined by \eqref{ec 1.1}, \eqref{ec 1.2}
and \eqref{ec 1.1}, \eqref{ec 1.3} are given by:
\begin{equation}\label{Def: u y ualpha}
\begin{array}{ll}
 u (x, y) = -\tfrac{1}{2} g x^2 + (g x_0 - q) x +b, \qquad \forall \;\; (x,y) \in \Omega   \\ 
 \\
 u_\alpha (x, y)= -\tfrac{1}{2} g x^2 + (g x_0 - q) x + \tfrac{1}{\alpha} (g\,x_0 - q) +b ,\qquad \forall \;\; (x,y) \in \Omega. 
 \end{array}
 \end{equation}

As 
a consequence of the symmetry of domain $\Omega$ {and the boundary conditions},  the solutions $u$ and $u_\alpha$ of systems $(S)$ and $(S_\alpha)$ are independent of variable $y$, and therefore, we  work  with one-dimensional problems.

      Given $n \in \mathbb{N}$, we define:
          \begin{equation} \label{ec 2.1}
    	    h= \tfrac{x_0}{n}; \quad x_i = (i-1) \,h, \text{ for } i=1,\ldots, n+1, \quad {u^h_i} \approx  u(x_i, y)  \text{ for } i=2, \ldots,n+1.
          \end{equation}

 {
Here, $n$ is the number of subintervals of $[0,x_0]$, $h$ is the uniform mesh size, $u^h_1=b$, and $u^h_i$ denotes the discrete approximation of the temperature at the node $x_i$, $i=2,\ldots,n+1$. Since the temperature is constant along the $y$-direction, $u^h_i$ approximates $u(x_i,y)$ for any $(x_i,y)\in\Omega$.
}

We apply the classical finite-difference method to the system $(S)$ described by \mbox{Equations} \eqref{ec 1.1} and \eqref{ec 1.2}.
Since the boundary condition on $\Gamma_1$ prescribes $u(0,y)=b$, we immediately obtain $u_1=b$.

For the interior nodes, we use the classical centered second-order finite-difference approximation:
\begin{equation}\label{aprox-deriv}
\frac{\partial^2 u}{\partial x^2}(x_i,y)\approx 
\frac{u(x_{i+1},y)-2u(x_i,y)+u(x_{i-1},y)}{h^2},
\qquad i=2,3,\ldots,n,
\end{equation}
and from {the differential equation} \eqref{ec 1.1}, {we impose that}
\begin{equation}\label{aprox-nodo-interior}
-gh^{2}= {u^h_{i+1}}-2{u^h_i}+{u^h_{i-1}},
\qquad i=2,3,\ldots,n.
\end{equation}

To incorporate the Neumann boundary condition on $\Gamma_2$,  we use a backward finite difference for the first derivative:
\begin{equation}\label{aprox-Neumann}
\frac{\partial u}{\partial x}(x_{n+1},y)\approx 
\frac{u(x_{n+1},y)-u(x_{n},y)}{h},
\end{equation}
which, using the boundary condition 
\(\frac{\partial u}{\partial x}(x_{n+1},y)=q\),
leads to {assuming that}
\begin{equation}\label{aprox-un+1}
-qh = {u^h_{n+1}} - {u^h_{n}}.
\end{equation}
{Taking into account  \eqref{aprox-nodo-interior} and \eqref{aprox-un+1}, the resulting discretization leads to the} discrete linear system ${(S^h)}$
\begin{equation}\label{sistema-tradicional}
{A\,v^h = T^h,}
\end{equation}
where ${v^h} = ({u^h_i})_{i=2,\ldots,n+1}\in\mathbb{R}^{n}$ {denotes} the vector of unknowns, $A$ is the associated tridiagonal coefficient matrix:

    					\begin{equation}\label{matrizA}
    					A=\begin{pmatrix}
-2 & 1 & 0 & \ldots & \ldots & & 0 &\\                                                                                                   
1 &-2 & 1&0 &\ldots &&0 \\
0 &1& -2 & 1&0&\ldots &0 \\
\vdots\\
0 &\ldots &0 & 1&-2 & 1 &0\\
0 &\ldots & \ldots & 0&1 & -2 & 1\\
 0 &\ldots& \ldots& \ldots &  0& -1 &1
\end{pmatrix}_{n\times n}
\end{equation} 
    and ${T^h}\in \mathbb{R}^n$ is  the vector of independent terms:
    \begin{equation}\label{vectorT}
     {T^h}= \Big(-g\,h^2 - b, -g h^2, \ldots, -g h^2, -q\,h  \Big)^t.\end{equation}

    \noindent The square matrix $A$ is invertible and its inverse matrix is given by
\[A^{-1}=\begin{pmatrix}
-1 & -1 & -1 & \ldots & \ldots & -1 & 1 &\\                                                                                                   
-1 & -2 & -2 & \ldots & \ldots & -2 & 2 &\\                                                                                                   
-1 & -2 & -3 & \ldots & \ldots & -3 & 3 &\\                                                                                                   
\vdots & \vdots & \vdots  &&& \vdots  & \vdots\\
-1& -2 & -3 &\ldots && -(n-1) & n-1\\
 -1 & -2& -3& \ldots &  &-(n-1) &n
\end{pmatrix}_{n\times n}.\]
Then, the linear system ${(S^h)}$ has a unique solution:
$$ {u^h_{i}}= b+  h^2g  \left( (i-1) \;n-\tfrac{i(i-1)}{2}\right)-(i-1) hq,\quad i=2 ,\ldots, n+1.$$   
As $n=\tfrac{x_0}{h}$ {and $u^h_1=b$}, it follows that:
    \[ \qquad\quad {u^h_i} = b + (i-1) \,h \,(g\,x_0-q) -  h^2 g \, \tfrac{\,i(i-1)\,}{2},\quad i=1, \ldots, n+1. \]

 Then, the continuous solution $u(x,y)$ of system $(S)$ can be approximated by the piecewise linear interpolant {$u^h(x,y)$} obtained from the nodal values computed by the finite difference scheme. More precisely, we define
          \begin{equation}\label{Def uh}
    	    {u^h}(x, y)= (g x_0  - q - h\, g\, i) x + h^2 g \Big(\tfrac{i(i-1)}{2}\Big) + b,\quad x \in [x_i, \,x_{i+1}], \,   \;y\in [0, y_0], 
          \end{equation}
with $i=1, \ldots , n.$

The following lemma{ shows that the discrete solution ${u^h}$ provides a first-order accurate approximation of the exact solution $u$ and its derivative with respect to $x$.}

  \begin{Lemma}\label{Cota u uh} $ $
    \begin{enumerate}   	
\item[(a)]  For every 
 {grid point} $(x_i,y)$ with $i=1,\ldots, n+1$, $y\in[0,y_0]$, {the following comparison holds}: 
    				\begin{itemize}
    						\item [(i)] if $g > 0$  then ${ u^h}(x_i,y) \leq u(x_i,y)$.
 
    						\item [(ii)] if $g < 0,$ then ${ u^h}(x_i,y) \geq u(x_i,y)$.
    				\end{itemize}    
    				
\item[(b)]  {The approximation error satisfies first-order estimates in the $H$-norm, namely,}
    		\[\lVert u-{ u^h} \rVert _H \leq  C_1\,h, \qquad \text{and}\qquad
\lVert \pder[u]{x}-\pder[{u^h}]{x} \rVert _H \leq \widetilde{C_{1}}\, h, \]
    		where {the  constants $C_1$ and $\widetilde{C_1}$, which do not depend on $h$,  are given by }
    		$C_1= x_0\,|g| \,\sqrt{\tfrac{2}{15}x_0\, y_0}$ and $\widetilde{C_{1}}= |g|\sqrt{\tfrac{1}{3}x_0 \, y_0}. $
    \end{enumerate}
    
    \end{Lemma}
\begin{proof} $ $

    \begin{itemize}
     \item[(a)] From the functions $u$ and ${ u^h}$, {given by \eqref{Def: u y ualpha} and \eqref{Def uh}, respectively}, we have 
    			\[ u(x_i, y) - { u^h}(x_i,y)=   \tfrac{g\,h^2 \,(i-1)}{2}, \quad i=1, \ldots , n+1.\] 
      
     \item[(b)] From the definition of the norm in space $H$ and Formulas \eqref{Def: u y ualpha} and \eqref{Def uh} for functions $u$ and ${ u^h}$, respectively, it follows that: 
     $$\begin{array}{ll}
     ||u-{ u^h}||_{H}^2& =y_0 \displaystyle\sum\limits_{i=1}^n \bint_{x_i}^{x_{i+1}}  \big(u(x,y)-{ u^h}(x,y)\big)^2  \dee x\\ 
     \\
    &=\frac{1}{120} y_0 \, h^5\, g^2\, n^3\,\big(\frac{1}{n^2}+\frac{5}{n}+10 \big)\\ 
    \\
        &\leq \frac{2}{15}\,y_0\, h^5\, g^2\, n^3=\frac{2}{15}\, h^2 g^2 \,x_0^3\,y_0=C_1^2 h^2.
     \end{array}$$
    \end{itemize}
The norm  $||\pder[u]{x}-\pder[{u^{h}}]{x}||_{H}$ can be computed analogously.
\end{proof}

  We next apply the classical finite-difference method to the system $(S_\alpha)$ defined by Equations \eqref{ec 1.1} and \eqref{ec 1.3}. {For each $n\in\mathbb{N}$, we set $h=\frac{x_0}{n}$ and denote by $u^h_{\alpha,i} \approx u_\alpha(x_i,y)$ the approximate value of $u_\alpha$ at $(x_i,y)$} for $i=1,\cdots,n+1$.

The Robin boundary condition on $\Gamma_1$ is approximated by a {classical forward finite-difference scheme}, namely,
\begin{equation}
{\frac{u_\alpha(x_2,y)-u_\alpha(x_1,y)}{h}\approx \frac{\partial u_\alpha}{\partial x}(x_1,y)}.
\end{equation}
{Taking into account that $\frac{\partial u_\alpha}{\partial x}(x_1,y)
=\alpha\bigl(u_\alpha(x_1,y)-b\bigr)$, we impose that}
\begin{equation}\label{Robin-Clasica}
{\frac{{u^h_{\alpha,2}}-{u^h_{\alpha,1}}}{h}
= \alpha \bigl({u^h_{\alpha,1}}-b\bigr).}
\end{equation}

Moreover, at the interior nodes we use the approximation given in \eqref{aprox-nodo-interior}, while the Neumann boundary condition at $x_{n+1}$ is discretized according to \eqref{aprox-Neumann}.

Then, we obtain the linear system ${(S^h_{\alpha})}$:
       \[A_\alpha \, {v^h_\alpha} = {T^h_\alpha},\]
 \noindent where the vector of unknowns ${v^h_\alpha} \in  \mathbb{R}^{n+1}$ is given by ${v^h_\alpha} = ({u^h_{\alpha,i}} )_{i= 1,\ldots n+1}$,  the tridiagonal coefficient matrix  $A_\alpha$  of order $n+1$ is defined as: 
  \begin{equation}\label{matrizAalpha}
     					A_{\alpha}=\begin{pmatrix}
                                                            -(1+\alpha\,h) & 1 & 0 & \ldots & \ldots & & 0 &\\                                                                                                   
                                                             1 &-2 & 1& \\
                                                             \vdots &
                                                             \ldots & \ldots & &1 & -2 & 1\\
                                                              0 & 0& 0& \ldots &  0& -1 &1
                                                    \end{pmatrix}_{(n+1)\times (n+1)}
                                                    \end{equation}
                                                    
  \noindent and \begin{equation}\label{Talpha}
  {T^h_\alpha}= \Big(-\alpha\,b\,h, -g h^2, \ldots, -g h^2, -q\,h  \Big)^t \in  \mathbb{R}^{n+1}.
  \end{equation}
     
 \noindent It can be seen that the square matrix $A_{\alpha}$ is invertible  and its inverse matrix is given by

\begin{scriptsize}
\[A_{\alpha}^{-1}=\frac{1}{\alpha h}\begin{pmatrix}
-1 & -1 & -1 & \ldots & \ldots & -1 & 1 &\\                                                                                                   
-1 & -(1+\alpha h) & -(1+\alpha h) & \ldots & \ldots & -(1+\alpha h) & 1+\alpha h &\\                                                                                                   
-1 & -(1+\alpha h) & -(1+2\alpha h) & \ldots & \ldots &-(1+2\alpha h) & -(1+2\alpha h) &\\                                                                                                   
\vdots & \vdots & \vdots  &&& \vdots  & \vdots\\
-1& -(1+\alpha h)& -(1+2\alpha h) &\ldots && -(1+(n-1)\alpha h) & 1+(n-1)\alpha h\\
 -1 & -(1+\alpha h)& -(1+2\alpha h)& \ldots &  &-(1+(n-1)\alpha h) &1+n\alpha h
\end{pmatrix}_{(n+1)\times (n+1)}\] 
\end{scriptsize}

Then, the linear system ${(S^h_{\alpha})}$ has a unique solution:
\[ {u^h_{\alpha,\,i}} = (b + \tfrac{g\,x_0 - q}{\alpha}) + (i-1)  \,(g\,x_0-q)\; h - \tfrac{g}{\alpha}\,h  -   g \, \tfrac{\,i(i-1)\,}{2}\,h^2,\quad i=1, \ldots, n+1. \]
      
 {As a consequence, the continuous solution $u_{\alpha}(x, y)$ of system $(S_\alpha)$ given by \eqref{Def: u y ualpha} can be approximated in $\overline{\Omega}$ by the discrete function $u^h_{\alpha}(x, y)$, defined as the piecewise linear interpolation of the nodal values obtained from the finite-difference system $(S^h_{\alpha})$.
}
 \begin{equation}\label{ec 2.3}
     	   { u^h_{\alpha}} (x, y) = ( g\, x_0 - q - h\, g\, i) x + 
     	    b +  \tfrac{g\,x_0 - q}{\alpha} -  \tfrac{g\,h}{\alpha} +  \tfrac{i^2-i}{2} g\,h^2 ,
           \end{equation}
     	for $x \in [x_i, \,x_{i+1}], y\in [0, y_0]$, $i=1, \ldots , n.$

Notice that  ${u^h_{\alpha}(x,y)}\to u^h(x,y)$ when $\alpha\to \infty$  for every $(x,y)\in \overline{\Omega}$.

\begin{Lemma}\label{lema convergencia uhalpha a ualpha}
{Let} $u_{\alpha}$ {be}  the solution of {problem }$(S_\alpha)$, {where $\alpha>0$ is the convective heat transfer coefficient appearing in the Robin boundary condition, and let $u^h_{\alpha}$ denote its piecewise linear discrete approximation defined} in (\ref{ec 2.3}). {Then, for each mesh size}  $h$, {the following error estimates hold}:
     		\[\lVert u_\alpha-{u^h_{\alpha}}  \rVert _H \leq C_{1\alpha}\,h,\qquad \qquad \text{and}\qquad
\lVert \tfrac{\partial u_{\alpha }}{\partial x}-\tfrac{\partial {u^h_{\alpha}}}{\partial x} \rVert _H \leq \widetilde{C_{1}}\, h, \]
     		{where } $C_{1\alpha}= \lvert\,g\rvert x_0\sqrt{ x_0\,y_0 \Big( \frac{2}{15} +\frac{2}{3}\frac{1}{\alpha x_0} + \tfrac{1}{\alpha^2 x_0^2} \Big)  }$ and $\widetilde{C_{1}}= |g|\sqrt{\tfrac{1}{3}x_0 \, y_0} $ {are positive constants independent of $h$. }
     	
\end{Lemma}

\begin{proof}
{Taking into account that $u_\alpha$ and {$u^h_{\alpha}$} are given by \eqref{Def: u y ualpha} and \eqref{ec 2.3}, respectively}, we have
 $$\begin{array}{ll}
&     ||u_\alpha-{u^h_{\alpha}}||_{H}^2 =y_0 \displaystyle\sum\limits_{i=1}^n \bint_{x_i}^{x_{i+1}}  \big(u_\alpha(x,y)-{u^h_{\alpha}}(x,y)\big)^2  \dee x\\ 
    &=y_0 \, g^2 \displaystyle\sum\limits_{i=1}^n  \, \frac{h^5 i^2}{4}-\frac{h^5 i}{6}+\frac{h^5}{20}-\frac{h^4}{3 \alpha }+\frac{h^4 i}{\alpha }+\frac{h^3}{\alpha ^2}\\ 
   &= y_0\, g^2\, \Big[ h^5 n^3 \left( \frac{1}{12}+\frac{1}{24n}+\frac{1}{120n^2}\right) +\frac{h^4 n^2}{\alpha}\left( \frac{1}{2}+\frac{1}{6n}\right)+\frac{h^3n}{\alpha^2} \Big]\\
 &  \leq x_0\,  y_0\, g^2\,h^2 \, \Big( \frac{2}{15}x_0^2+\frac{2}{3} \frac{x_0}{\alpha}+\frac{1}{\alpha^2} \Big)=C_{1\alpha}^2 h^2.
     \end{array}$$
In addition, the {partial derivatives with respect to $x$} of functions $u_\alpha$ and {$u^h_{\alpha}$} are given by:
$$\pder[u_{\alpha}]{x}(x,y)=\pder[u]{x}(x,y)=-gx+gx_0-q,\qquad \forall (x,y)\in \Omega$$ 
   and
    $$\pder[{u^h_{\alpha}}]{x}(x,y)=\pder[{u^{h}}]{x}(x,y)=  g\; x_0-q-h\,g\,i, \qquad x\in [x_i,\, x_{i+1}],\quad y\in\,[0,y_0].$$
Then, the bound for $\lVert \pder[u_{\alpha}]{x}-\pder[{u^h_{\alpha}}]{x} \rVert _H$  coincides with the bound for $\lVert \pder[u]{x}-\pder[{u^h}]{x} \rVert _H$,  obtained in Lemma \ref{Cota u uh}.
\end{proof}     

\begin{Remark}
Notice that $C_{1\alpha}\to C_1$ when $\alpha\to \infty$, where $C_1$ is {the constant appearing} in Lemma~\ref{Cota u uh}. This shows that the error bound associated with the convective boundary condition converges to the one obtained for the Dirichlet problem as $\alpha \to \infty$.

\end{Remark}

\section{Distributed Optimization Problem with Variable \boldmath{$g$}}\label{sec3}   
In this section we obtain  discrete optimal  solutions to the continuous optimization problem $(P_1)$ and $(P_{1\alpha})$ in the rectangular domain $\Omega$ for the case where the optimization variable is $g$.

\subsection{Discrete Problem  ${(P^h_{1})}$ Associated with $(P_1)$}\label{sec3.1}
 
 Taking into account that $b, \,g,\,q$ and  the desired state $z_d$  in \eqref{J1-J1alpha} are constants, according to~\cite{BoGaTa2020}, the continuous quadratic  functional cost for problem $(P_1)$ is explicitly given by:
\begin{equation} \label{J1-explicito}   
    J_1(g)= \tfrac{1}{2} x_0^3\, y_0\, q^2 \Big\{ g^2 \,\tfrac{x_0^2}{q^2} \Big(\tfrac{2}{15}  + \tfrac{M_1}{x_0^4} \Big) + g \tfrac{x_0}{q}  \Big(- \tfrac{5}{12}  + \tfrac{2}{3} \tfrac{(b- z_d)}{q x_0} \Big) + \\
    \tfrac{1}{3} - \tfrac{(b-z_d)}{q x_0} + \tfrac{(b-z_d)^2}{q^2x_0^2} \Big\}.
    \end{equation}
    
    
\noindent Then, the solution to the distributed optimization problem $(P_1)$ is $g_{op}$ defined by:
     \begin{equation}\label{gop}
          g_{op} = \frac{q}{3\, x_0} \frac{\Big( \tfrac{5}{8 } -\tfrac{ (b -z_d)}{q \, x_0}\Big)}{\Big(\tfrac{2}{15} + \tfrac{M_1}{x_0^4} \Big)}
     \end{equation}
   \noindent and the continuous optimization state when $g=g_{op}$ is     
       \begin{equation}\label{ugop}
            u_{g_{op}}(x,y) = -\tfrac{1}{2} g_{op}\, x^2 + (g_{op}\, x_0 - q) x +b.
       \end{equation}

We define the discrete distributed optimization problem ${(P^h_{1})}$ for the constant internal energy $g$ as 
$$ \text{ find }\quad {g^h_{op}}\in \mathbb{R}\quad\text{ such that }\quad
{J^h_1(g^h_{op})}=\min\limits_{g\in \mathbb{R}} \text{ }{J^h_{1}(g)}
$$
where the discrete cost function  ${J^h_{1}}$ is given by:
    \[{J^h_{1} }(g)=  \tfrac{1}{2} \lVert u^h_g - z_d \rVert^2_H +  \tfrac{1}{2} M_1 \lVert g \rVert ^2_{H} .\]
    {Here, $u^h_g$ denotes the discrete approximation corresponding to the internal energy $g$ (see~\eqref{Def uh}), $h$ is the discretization step defined in \eqref{ec 2.1}, $z_d$ is the desired target, and $M_1>0$ is a regularization parameter. }

    Taking into account that the variable $g$ is constant results in:
     
\begin{equation}
{J^h_{1}} (g)= \tfrac{1}{2} y_0 \Big\{M_1\,g^2 \,x_0 + \, \sum_{i=1}^n \, \int_{x_i}^{x_{i+1}} \left({u^h_g (x,y) }- z_d\right)^2 \,\dee x \Big \}
\end{equation}
 and from algebraic work, it follows that
\begin{equation}\label{J1h-explicito}   \begin{array}{rl}
    {J^h_{1}(g)}&= \frac{1}{2} x_0^3\, y_0 \,q^2 \Bigg\{g^2 \tfrac{x_0^2}{q^2} \Big[\frac{2}{15}+\tfrac{M_1}{x_0^4}+\frac{1}{180} \tfrac{h^4}{x_0^4}+\frac{1}{24}\tfrac{h^3}{x_0^3}+\frac{1}{36}\tfrac{h^2}{x_0^2}-\frac{5 }{24} \tfrac{h}{x_0}\Big] \\[0.35cm]
&    + g\,\tfrac{x_0}{q} \Big[- \frac{5}{12}+ \frac{2}{3} \tfrac{(b-z_d)}{q x_0}   + \tfrac{h}{x_0}  \Big( \frac{1}{3} +  \frac{1}{12}\tfrac{h}{x_0} \Big) -  \tfrac{h}{x_0} \,\tfrac{(b-z_d)}{qx_0} \Big(  \frac{1}{2} + \tfrac{1}{6}\tfrac{h}{x_0} \Big) \Big] \\[0.35cm]
 &   + \frac{1}{3} - \tfrac{(b-z_d)}{q\,x_0}  + \tfrac{(b-z_d)^2 }{q^2\,x_0^2}\Bigg\}. \end{array}
 \end{equation}

\begin{Lemma}    \label{Lema:J1(g)-J1h(g)}
{For any given internal energy $g \in \mathbb{R}$, the following estimate holds for the discrete cost functional $J^h_{1}$:}
    \begin{equation}\label{ec 3.3}
       \lvert J_1(g) - {J^h_{1}}(g)\rvert  \approx C_2 \, h 
    \end{equation}
   where  $C_2= \frac{1}{2}x_0^3\, y_0\,g\;  q \Big|-\frac{5}{24} \tfrac{g x_0}{q}+\tfrac{1}{3}-\tfrac{1}{2} \tfrac{(b-z_d)}{qx_0}\Big|$  {is a constant independent of h}.
    \end{Lemma}

\begin{proof}
     From \eqref{J1-explicito} and \eqref{J1h-explicito} we get
   $$   \begin{array}{rl}
 {J^h_1}(g)-J_1(g)&= \frac{1}{2} x_0^3\, y_0 \,q^2 \Big\{g^2 \tfrac{x_0^2}{q^2} \Big[\frac{1}{180} \tfrac{h^4}{x_0^4}+\frac{1}{24}\tfrac{h^3}{x_0^3}+\frac{1}{36}\tfrac{h^2}{x_0^2}-\frac{5 }{24} \tfrac{h}{x_0}\Big] \\[0.45cm]
&    + g\,\tfrac{h}{q}  \Big[   \frac{1}{3} +  \frac{1}{12}\tfrac{h}{x_0} - \,\tfrac{(b-z_d)}{qx_0} \Big(  \frac{1}{2} + \tfrac{1}{6}\tfrac{h}{x_0} \Big) \Big] \Big\}\\[0.35cm]
&    \approx \frac{1}{2}x_0^3\, y_0\, q^2 \Big[-\frac{5}{24} \tfrac{g^2 x_0}{q^2}+\tfrac{g}{q}\Big(\tfrac{1}{3}-\tfrac{1}{2} \tfrac{(b-z_d)}{qx_0}\Big)\Big] h\\[0.45cm] 
&= \frac{1}{2}x_0^3\, y_0\,g q  \Big[-\frac{5}{24} \tfrac{g x_0}{q}+\tfrac{1}{3}-\tfrac{1}{2} \tfrac{(b-z_d)}{qx_0}\Big] h.
    \end{array}$$
    Therefore, we obtain \eqref{ec 3.3}.
\end{proof}

     From the optimality condition we obtain the following result:

   \begin{Lemma}\label{lema:goph} $ $
    
    \begin{itemize}
     \item[(a)] The explicit expression for the optimal variable ${g^h_{op}}$ is given by: 
     
    \begin{equation}\label{ec 3.4}
{g^h_{op}} = \tfrac{q}{3\,x_0} \tfrac{A_1 + \tfrac{h}{x_0}\,A_2 + \tfrac{h^2}{x_0^2}\, A_3}{A_4 + A_5(h)},
    \end{equation}
      where
\begin{equation}\label{Ai}
\begin{array}{ll}
A_1=\tfrac{5}{8} - \tfrac{b- z_d}{qx_0},\qquad \qquad & A_4 = \tfrac{2}{15} + \tfrac{M_1}{x_0^4},\\[0.15cm]
 A_2=\tfrac{3 (b-z_d)}{4 q x_0}  -\tfrac{1}{2}, & A_5(h) =  \tfrac{h}{12\,x_0} \Big( \tfrac{h^3}{15\, x_0^3} + \tfrac{h^2}{2 \, x_0^2} + \tfrac{h}{3\,x_0} - \tfrac{5}{2} \Big).\\[0.15cm]
  A_3=  \tfrac{b-z_d}{4\,qx_0} - \tfrac{1}{8}, \\ 
\end{array}
\end{equation}
      
     \item[(b)]In addition, the following error estimates hold:

  \begin{equation}\label{ec 3.5}
    				      \lvert g_{op} - {g^h_{op}}\rvert \approx C_3 \, h ,
    				  \end{equation}
    
\begin{equation}\label{ec 3.6}
    				     \Big \lvert J_1(g_{op}) - {J^h_{1}(g^h_{op})} \Big \rvert \approx C_4 \, h ,
    				\end{equation}
    where $C_3$  and $C_4$ do not depend on $h$.
   
    \end{itemize}
    \end{Lemma}

  \begin{proof} $ $

    \begin{itemize}
    \item[(a)] { It follows immediately from the expression of the derivative of $J^h_1$ with respect to $g$.}
    
\item[(b)] Rewriting $g_{op}$ given by \eqref{gop} as:
    	  $g_{op}= \tfrac{q}{3\,x_0} \tfrac{A_1 }{A_4},$
    	  it follows that:
    	  \[   {g^h_{op}}-g_{op} = \tfrac{q}{3\,x_0}  \tfrac {-A_1\,A_5(h)+\,\,\tfrac{h}{x_0} A_2\, A_4 + \tfrac{h^2}{x_0^2}\,A_3\,A_4   }{A_4^2 + A_4 \,A_5(h)\,\,} \approx\tfrac{q}{3\,x_0^2} \tfrac{  A_2A_4+\tfrac{5}{24}A_1 }{ A_4^2} h+o(h^2),\]
    and we obtain (\ref{ec 3.5}) with $C_3=|C_3^*|$ where \begin{equation}\label{C3*}
    C_3^*=\tfrac{q}{3\,x_0^2} \tfrac{ A_2 \,A_4+ \tfrac{5}{24}A_1 }{A_4^2}. 
    \end{equation}
  
   From the expressions for $J_1(g)$ at $g=g_{op}$ and ${J^h_{1}(g)}$ at ${ g= g^h_{op}}$, it follows that:
    \[\begin{array}{ll}
&    J_1(g_{op})- {J^h_{1}(g^h_{op})} =\tfrac{1}{2}x_0^3\;y_0\; q^2 \bigg[\tfrac{x_0^2}{q^2}A_4(g_{op} ^2-({g^{h}_{op}})^2)  -\tfrac{2}{3}\tfrac{x_0}{q} A_1 (g_{op}-{g^h_{op}})\\
    & -({g^h_{op}})^2 \tfrac{x_0^2}{q^2}A_5(h)+\tfrac{2}{3} {g^h_{op}} \tfrac{h}{q} \left( A_2+\tfrac{h}{x_0}A_3\right)\bigg]. 
\end{array}\]
   \noindent { By using  (\ref{ec 3.4}) and \eqref{ec 3.5} we get \eqref{ec 3.6} where }
   $$\begin{array}{ll}
   C_4 &= \tfrac{1}{2} x_0^2 y_0 q^2  \left|-2A_4 C_3^* g_{op}\tfrac{x_0^3}{q^2} +\tfrac{5}{24} g_{op}^2\tfrac{x_0^2}{ q^2}+\tfrac{2}{3}A_1 C_3^* \tfrac{x_0^2}{q}+\tfrac{2}{3} A_2 g_{op} \tfrac{x_0}{q} \right| \\
   \\
   &= \tfrac{1}{2} x_0^2 y_0 q^2  \left|\tfrac{A_1(5A_1+48 A_2A_4)}{216 A_4^2}\right|.
   \end{array}$$

    \end{itemize}   
   \end{proof}
     
\begin{Lemma}\label{lema-dif norm ugop-uhghop}
 Let us consider $u_{g_{op}}$ the solution of { the system $(S)$ given by } \eqref{ec 1.1} and \eqref{ec 1.2} for $g= g_{op}$ and {$u^h_{g^h_{op}}$}
 the discrete solution defined by (\ref{Def uh}) for $h>0$ and for {$g=g^h_{op}$, where $g^h_{op}$}  is  the optimal value  of the problem ${(P^h_{1})}$ given by
 (\ref{ec 3.4}). We have: 
  $$   (a) \quad
  \lVert u_{g_{op}}- {u^h_{g^h_{op}}}    \rVert _H \approx C_5\, h,\qquad \qquad
  			        (b) \quad
  \left\lVert \pder[u_{g_{op}}]{x}-\pder[ {u^h_{g^h_{op}}}   ]{x}     \right\rVert _H \approx C_6\, h 
   $$
 \noindent where $C_5$ and $C_6$ are {positive constants that are independent of parameter $h$.}
   \end{Lemma}  

   \begin{proof}$ $
    \begin{itemize}
     \item[(a)] From the definition of the norm in $H$, we obtain 
      $$\begin{array}{ll}
   &  ||u_{g_{op}}-{u^h_{g^h_{op}}}   ||_{H}^2 =y_0 \displaystyle\sum\limits_{i=1}^n \bint_{x_i}^{x_{i+1}}  \Big(u_{g_{op}}(x,y)-{u^h_{g^h_{op}}}   (x,y)\Big)^2  \dee x\\ 
  & =y_0 \displaystyle\sum\limits_{i=1}^n \bint_{x_i}^{x_{i+1}} \Big[-\tfrac{1}{2} g_{op}x^2+x_0 x (g_{op}-{g^h_{op}})+ h {g^h_{op}}(i x-h\tfrac{i(i-1)}{2})\Big]^2  \dee x\\
  &= \tfrac{1}{120}x_0^3 \;y_0\; g_{op}^2 \;\Big[10-25 \; \;\tfrac{x_0 \; C_3^*}{g_{op}}+16 \left( \;\tfrac{x_0 \; C_3^*}{g_{op}}\right)^2\Big] h^2+ o(h^3).
     \end{array}$$
     where $C_3^*$ is given by \eqref{C3*}.
    Therefore, it follows that 
   $$  \lVert u_{g_{op}}- { u^h_{g^h_{op}}}    \rVert _H \approx C_5\, h,$$
   with $$C_5= |g_{op}| \sqrt{\tfrac{1}{120}x_0^3 \;y_0\;  \;\Big[10-25 \; \;\tfrac{x_0 \; C_3^*}{g_{op}}+16 \left( \;\tfrac{x_0 \; C_3^*}{g_{op}}\right)^2\Big]}.$$

     \item[(b)] We have

      $\begin{array}{ll}
    &  \Big\lVert \pder[u_{g_{op}}]{x}-\pder[ {u^h_{g^h_{op}}} ]{x} \Big\rVert  _H^2 \\[0.45cm]
    & =y_0 \displaystyle\sum\limits_{i=1}^n \bint_{x_i}^{x_{i+1}}  \Big(-g_{op}x+h {g^h_{op}}i+x_0(g_{op}-{g^h_{op}})\Big)^2  \dee x\\[0.45cm]
&     =  \tfrac{x_0 y_0}{6} \;\Big[2 g_{op}^2 x_0^2+g_{op}\;{g^h_{op}}(h^2+3hx_0-4x_0^2)+({g^h_{op}})^2(h^2-3hx_0+2x_0^2) \Big]\\[0.45cm]
&=\tfrac{x_0 y_0}{6} \Big[2 x_0^2 (g_{op} -{g^h_{op}})^2 + 
       3 h \;x_0\; {g^h_{op}}  \;(g_{op} - {g^h_{op}}) + h^2 {g^h_{op}}(g_{op} + {g^h_{op}})\Big]. 
        \end{array}$

     \noindent Taking into account Lemma \ref{lema:goph}, we get
     \[ \Big\lVert\pder[ u_{g_{op}}]{x}- \pder[{ u^h_{g^h_{op}}}]{x}    \Big\rVert  _H^2  =  \tfrac{x_0 y_0}{6} \;g_{op}^2\;  \Big[ 2 +3\;\tfrac{x_0 C_3^*}{g_{op}}+2 \left(\tfrac{x_0 C_3^*}{g_{op}}\right)^2 \Big] \; h^2+o(h^3).\]

     \noindent Then, $$ \Big\lVert\pder[ u_{g_{op}}]{x}- \pder[{ u^h_{g^h_{op}}}]{x}    \Big\rVert  _H \approx C_6 h,$$ with 
     $$C_6 = |g_{op}|\sqrt{ \tfrac{x_0 y_0}{6}   \Big[ 2 -3\;\tfrac{x_0 C_3^*}{g_{op}}+2 \left(\tfrac{x_0 C_3^*}{g_{op}}\right)^2 \Big] },$$
      where $C_3^*$ is given by \eqref{C3*}. 

     \end{itemize} \end{proof}

\subsection{Discrete Problem  ${(P^h_{1\alpha})}$ Associated with $(P_{1\alpha})$}

From~\cite{BoGaTa2020}, we know that  the continuous quadratic functional cost in \eqref{J1-J1alpha} for the optimization problem $(P_{1\alpha})$ is explicitly given by: 
      \begin{equation}\label{J1alpha-explicito}
      \begin{array}{ll}
          J_{1\alpha}(g)= J_1(g) \\
          + \tfrac{x_0^2 y_0 q^2}{2\alpha} \Bigg\{\tfrac{g^2 x_0^2}{q^2} \Big(\tfrac{2 }{3} + \tfrac{1}{\alpha x_0} \Big) + 
          \tfrac{g x_0}{q} \Big( - \tfrac{5 }{3} - \tfrac{2}{\alpha x_0} +\tfrac{2 (b-z_d)}{q x_0}\Big) + 1+\tfrac{1}{\alpha x_0}-\tfrac{2(b-z_d)}{qx_0}\Big) \Bigg\},
          \end{array}
      \end{equation}
      where $J_1$ is defined by \eqref{J1-explicito}.
      Moreover, the continuous optimal distributed variable denoted by $g_{\alpha_{op}}$ is {given by}
\begin{equation}\label{galphaop}
          g_{\alpha_{op}} = \frac{q}{3\, x_0} \frac{\Big( \tfrac{5}{8 } -\tfrac{ (b -z_d)}{q \, x_0}+\tfrac{5}{2\alpha x_0}+\tfrac{3}{\alpha^2x_0^2}-\tfrac{3(b-z_d)}{\alpha q x_0^2}\Big)}{\Big(\tfrac{2}{15} + \tfrac{M_1}{x_0^4}+\tfrac{2}{3\alpha x_0}+\tfrac{1}{\alpha^2 x_0^2} \Big)}.
     \end{equation}
\noindent The continuous associated state is established by: 
     \begin{equation}\label{ualpha-galphaop}
            u_{\alpha g_{\alpha_{op}}}(x,y) = -\tfrac{1}{2} g_{\alpha_{op}} x^2 + (g_{\alpha_{op}} x_0 - q) x + \tfrac{1}{\alpha} (g_{\alpha_{op}}\,x_0 - q) +b.
       \end{equation}

       \noindent {We define} the discrete cost function as
\begin{equation}    
  {J^h_{1\alpha}(g)} = \frac{1}{2} \lVert {u^h_{\alpha g}} - z_d \rVert_H^2 + \frac{1}{2} M_1 \lVert g \rVert_H^2,
\end{equation}
where function ${u^h_{\alpha g}}$, given in \eqref{ec 2.3}, {denotes the discrete approximation corresponding to the internal energy $g$, $h>0$ is the discretization step, $z_d$ is the desired target, and $M_1>0$ is a constant parameter.} We set the following discrete optimization problem
 ${(P^h_{1\alpha})}$ on the constant internal energy $g$ as
$$ \text{find }\quad {g^h_{\alpha_{op}}}\in \mathbb{R}\quad\text{ such that }\quad
{J^h_{1\alpha}(g^h_{\alpha_{op}})}=\min\limits_{g\in \mathbb{R}} {\text{ }J^h_{1 \alpha }(g).}
$$
The discrete cost function  ${J^h_{1 \alpha}}$  is explicitly given by
\begin{equation}\label{J1halpha-explicito}
\begin{array}{ll}
&{J^h_{1 \alpha}} ( g)  = J_{1\alpha} ( g )\\
& +\tfrac{1}{2} x_0^3\, y_0 \,g\, q\, h \Bigg\{ \tfrac{g x_0}{q} \Big[ - \tfrac{5 }{24}+ \tfrac{1}{36}\tfrac{h}{x_0} + \tfrac{1}{24}\tfrac{h^2}{x_0^2} + \tfrac{1}{180}\tfrac{h^3}{x_0^3} +\tfrac{1}{\alpha x_0}\left( - \tfrac{7 }{6 }+\tfrac{1}{3}\tfrac{h}{x_0}+\tfrac{1}{6}\tfrac{h^2}{x_0^2}\right)\\ 
& +\tfrac{1}{\alpha^2 x_0^2}\left( -2+\tfrac{h}{x_0}\right)\Big]
+ \tfrac{1}{3}+  \tfrac{1}{12}\tfrac{h}{x_0}+\tfrac{1}{\alpha x_0}\left( \tfrac{3}{2}+\tfrac{1}{6}\tfrac{h}{x_0}\right) +\tfrac{2}{\alpha^2 x_0^2}\\
&+\tfrac{(b-z_d)}{q x_0}\left(-\tfrac{1}{2}-\tfrac{1}{6}\tfrac{h}{x_0}-\tfrac{2}{\alpha x_0} \right) \Bigg\},
\end{array}
\end{equation}
where $J_{1\alpha}$ is given by \eqref{J1alpha-explicito}.

\begin{Lemma} For $g\in H, \; h>0$ and $\alpha>0$, the following estimate holds

  $$     \lvert  {J^h_{1 \alpha}}(g) - J_{1\alpha}(g) \rvert  \approx C_{2 \alpha} \, h $$
   where \[C_{2 \alpha} = \tfrac{1}{2} x_0^3\, y_0 \,|g|\, \Bigg| \tfrac{g x_0}{q} \Big( - \tfrac{5 }{24} - \tfrac{7 }{6 }\tfrac{1}{\alpha x_0} -\tfrac{2}{\alpha^2 x_0^2}\Big)+ \tfrac{1}{3}+  \tfrac{3}{2}\tfrac{1}{\alpha x_0}+\tfrac{2}{\alpha^2 x_0^2}+\tfrac{(b-z_d)}{q x_0}\left(-\tfrac{1}{2}-\tfrac{2}{\alpha x_0} \right) \Bigg|,
\] 
   is  a constant independent of $h.$
      
    \end{Lemma}
\begin{proof}
It follows immediately from expression \eqref{J1halpha-explicito}.
\end{proof}

\begin{Remark} 
 $C_{2\alpha}\to C_2$ when $\alpha\to \infty$, where $C_2$ is given in Lemma \ref{Lema:J1(g)-J1h(g)}.
\end{Remark}

   \begin{Lemma} $ $

    \begin{itemize}
     \item[(a)] The explicit expression for the optimal { control $g^h_{\alpha_{op}}$} is given by: 
      \begin{equation}\label{galphaoph}
     {g^h_{\alpha_{op}}} = \tfrac{q}{3\,x_0} \tfrac{A_{1\alpha} + \tfrac{h}{x_0}\,A_{2\alpha} + \tfrac{h^2}{x_0^2}\, A_{3\alpha}}{A_{4\alpha} + A_{5\alpha}(h)},
      \end{equation}
  where
\begin{equation}\label{Aialpha}
\begin{array}{rl}
A_{1\alpha}&=\tfrac{5}{8} - \tfrac{b- z_d}{qx_0}+\tfrac{1}{\alpha x_0}\bigg(\tfrac{5}{2}+\tfrac{3}{\alpha x_0}-\tfrac{3(b-z_d)}{ q x_0}\bigg), \\[0.35cm]
 A_{2\alpha}&=\tfrac{3 (b-z_d)}{4 q x_0}  -\tfrac{1}{2}+\tfrac{1}{\alpha x_0}\bigg(-\tfrac{9}{4}-\tfrac{3}{\alpha x_0}+\tfrac{3(b-z_d)}{qx_0} \bigg),  \\[0.35cm]
  A_{3\alpha}&=  \tfrac{b-z_d}{4\,qx_0} - \tfrac{1}{8}-\tfrac{1}{4\alpha x_0}, \\[0.35cm] 
   A_{4\alpha} &= \tfrac{2}{15} + \tfrac{M_1}{x_0^4}+\tfrac{1}{\alpha x_0}\bigg( \tfrac{2}{3}+\tfrac{1}{\alpha x_0}\bigg),\\[0.35cm]
   A_{5\alpha}(h) &=  \tfrac{h}{12\,x_0} \Big( \tfrac{h^3}{15\, x_0^3} + \tfrac{h^2}{2 \, x_0^2} + \tfrac{h}{3\,x_0} - \tfrac{5}{2} \Big)\\
    &+\tfrac{h}{\alpha x_0^2}\bigg( -\tfrac{7}{6}+\tfrac{h}{3x_0}+\tfrac{h^2}{6x_0^2}+\tfrac{1}{\alpha x_0}\left( -2+\tfrac{h}{x_0}\right)\bigg).
\end{array}
\end{equation}  	  
    
     \item[(b)]In addition, the following error estimates hold:

  \begin{equation}\label{resta ghalphaop-galphaop}
    				      \lvert g_{\alpha_{op}} -      {g^h_{\alpha_{op}}}\rvert \approx C_{3\alpha} \, h, 
    				  \end{equation}
    
\begin{equation}\label{dif J1alpha}
    				     \Big \lvert J_{1\alpha}(g_{\alpha_{op}}) -      { J_{1\alpha}^h(g^h_{\alpha_{op}})} \Big \rvert \approx C_{4\alpha} \, h ,
    				\end{equation}
    
    where  $C_{3\alpha}$  and $C_{4\alpha}$ do not depend on  $h$.
   
    \end{itemize}
    \end{Lemma}

    \begin{proof}
     $ $
     \begin{itemize}
     \item[(a)] It follows immediately from the fact that
     $${J^h_{1\alpha}(g)}=\tfrac{1}{2}x_0^3 y_0
 q^2\Big[  2g \tfrac{x_0^2}{q^2}\Big( A_{4\alpha} + A_{5\alpha}(h)\Big) - \tfrac{2}{3} \tfrac{x_0}{
      q} \Big( A_{1\alpha} + \tfrac{h}{x_0} A_{2\alpha} + \tfrac{h^2}{x_0^2} A_{3\alpha}\Big)\Big].$$

\item[(b)] Notice that $g_{\alpha_{op}}$ given by \eqref{galphaop} can be rewriten as $g_{\alpha_{op}}=\tfrac{q}{3x_0} \tfrac{A_{1\alpha}}{A_{4\alpha}}.$
Then,
 \[ \begin{array}{ll}
        {g^h_{\alpha_{op}}}-g_{\alpha_{op}} & = \tfrac{q}{3\,x_0}  \tfrac {-A_{1\alpha}\,A_{5\alpha}(h)+\,\,\tfrac{h}{x_0} A_{2\alpha}\, A_{4\alpha} + \tfrac{h^2}{x_0^2}\,A_{3\alpha}\,A_{4\alpha}   }{A_{4\alpha}^2 + A_{4\alpha} \,A_{5\alpha}(h)\,\,}\\[0.45cm]
   & \approx\tfrac{q}{3\,x_0^2} \tfrac{  A_{2\alpha}A_{4\alpha}+\tfrac{5}{24}A_{1\alpha} }{ A_{4\alpha}^2} h+o(h^2).
   \end{array}\]
    and we obtain \eqref{resta ghalphaop-galphaop} with $C_{3\alpha}=|C_{3\alpha}^*|$, where \begin{equation}\label{C3alpha*}
    C_{3\alpha}^*=\tfrac{q}{3\,x_0^2} \tfrac{ A_{2\alpha} \,A_{4\alpha}+ \left(\tfrac{5}{24}+\tfrac{7}{6 \alpha x_0}\right)A_{1\alpha} }{A_{4\alpha}^2}. 
    \end{equation}
 Following  Lemma \ref{lema:goph} we obtain formula \eqref{dif J1alpha} with 
 $$\begin{array}{ll}
 C_{4\alpha} &= \tfrac{1}{2} x_0^2 y_0 q^2  \Big|-2A_{4\alpha} C_{3\alpha}^* g_{\alpha_{op}}\tfrac{x_0^3}{q^2} +\tfrac{5}{24} g_{\alpha_{op}}^2\tfrac{x_0^2}{ q^2}\\[0.45cm]
 &+\tfrac{7}{6} g_{\alpha_{op}}^2\tfrac{x_0}{ \alpha q^2}+\tfrac{2}{3}A_{1\alpha} C_{3\alpha}^* \tfrac{x_0^2}{q}+\tfrac{2}{3} A_{2\alpha} g_{\alpha_{op}} \tfrac{x_0}{q} \Big|.
 \end{array}
$$
     \end{itemize}
    \end{proof}
\begin{Remark} When $\alpha\to \infty$, we have 
     $A_{i\alpha}\to A_i$,  where $A_i$ and $A_{i\alpha}$ are given by \eqref{Ai} and \eqref{Aialpha}, respectively, for $i=1,2,\cdots, 5$.
   As an immediate consequence it follows that $C_{3\alpha}\to C_3$ and $C_{4\alpha}\to C_4$ when $\alpha\to\infty$, where $C_3$ and $C_4$ are defined in Lemma \ref{lema:goph}.
\end{Remark}

\begin{Lemma} Let us consider $u_{\alpha g_{\alpha_{op}}}$, the function given by \eqref{Def: u y ualpha} for $g= g_{\alpha_{op}}$ where $g_{\alpha_{op}}$ is the optimal variable of problem $(P_{1\alpha})$ given by \eqref{galphaop}, and ${u^h_{\alpha g^ h_{\alpha_{op}}}}$, the function defined by \eqref{ec 2.3} for $h>0$ where $g={g^h_{\alpha_{op}}}$ is the optimal {control} of {$(P^h_{1\alpha})$} given by
\eqref{galphaoph}. We have: 
  $$   (a) \quad
  \lVert u_{\alpha g_{\alpha_{op}}}- {u^h_{\alpha g^h_{\alpha_{op}}} }   \rVert _H \approx C_{5\alpha}\, h, \qquad
  			        (b) \quad
  \Big\lVert \pder[u_{\alpha g_{\alpha_{op}}}]{x}-\pder[ {u^h_{\alpha g^h_{\alpha_{op}}} }  ]{x}    \Big\rVert _H \approx C_{6\alpha}\, h ,
   $$
 \noindent where  $C_{5\alpha}$ and $C_{6\alpha}$ are positive constants independent of parameter $h.$
   \end{Lemma}  

   \begin{proof}
   Working algebraically we can obtain
   $$\begin{array}{ll}
   C_{5\alpha}= |g_{\alpha_{op}}| \Bigg\lbrace\tfrac{1}{120}x_0^3 \;y_0\;  \;\Big[10-25\tfrac{x_0 \; C_{3\alpha}^*}{g_{op}}+16 \left( \;\tfrac{x_0 \; C_{3\alpha}^*}{g_{op}}\right)^2\\
   \qquad +\tfrac{1}{\alpha x_0}\Big(60+\tfrac{120}{\alpha x_0}-240\tfrac{C_{3\alpha}^*}{\alpha g_{\alpha_{op}}}+120 \tfrac{C_{3\alpha}^* x_0}{\alpha g_{\alpha_{op}}^2}-140 \tfrac{C_{3\alpha}^* x_0}{ g_{\alpha_{op}}}+80 \tfrac{C_{3\alpha}^* x_0^2}{ g_{\alpha_{op}}^2} \Big)\Big]\Bigg\rbrace^{1/2},
   \end{array}$$
   and
   \[C_{6\alpha} = |g_{\alpha_{op}}|\sqrt{ \tfrac{x_0 y_0}{6}   \Big[ 2 -3\;\tfrac{x_0 C_{3\alpha}^*}{g_{\alpha_{op}}}+2 \left(\tfrac{x_0 C_{3\alpha}^*}{g_{\alpha_{op}}}\right)^2 \Big] },\]
      where $C_{3\alpha}^*$ is given by \eqref{C3alpha*}.
   \end{proof}
  \begin{Remark}  $C_{5\alpha}\to C_5$ and $C_{6\alpha}\to C_6$ when $\alpha\to\infty$, where $C_5$ and $C_6$ are given in Lemma \ref{lema-dif norm ugop-uhghop}.
  \end{Remark}

\begin{Remark} 
In~\cite{Ta2016} the double convergence when $(h,\alpha)\to (0,+\infty)$ of optimal control problem ${(P^h_{1\alpha})}$ was studied, obtaining a commutative diagram that relates the continuous and discrete optimal control problems $(P_{1})$, $(P_{1\alpha})$, {$(P^h_{1})$} and ${(P^h_{1\alpha})}$ as in the following scheme:
\begin{scriptsize}
\begin{center}
\begin{tikzpicture}
  \matrix (m) [matrix of math nodes,row sep=8em,column sep=8em,minimum width=2em]
  {
     \begin{array}{c} \text{Problem } ({P^h_{1}})\\ {g^h_{op}, J^h_{1}(g^h_{op}),  u^h_{ g^h_{op}}} \end{array}& \begin{array}{c} \text{Problem } (P_{1})\\ g_{op}, J_{1}(g_{op}),  u_{g_{op}} \end{array} \\
   \begin{array}{c} \text{Problem } ({P^h_{1\alpha}})\\ {g^h_{\alpha_{op}}, J^h_{1\alpha}(g^h_{\alpha_{op}}),  u^h_{\alpha g^h_{\alpha_{op}}}} \end{array}& \begin{array}{c} \text{Problem } (P_{1\alpha})\\ g_{\alpha_{op}}, J_{1\alpha}(g_{\alpha_{op}}),  u_{\alpha g_{\alpha_{op}}} \end{array}\\};
  \path[-stealth]
      	 (m-2-1) edge node [left] {$\alpha\to \infty$} (m-1-1)
       
        (m-2-2) edge node [right] {$\alpha\to \infty$} (m-1-2)
        
        (m-1-1) edge node [below] {$h\to 0$} (m-1-2)
        
          (m-2-1) edge [dashed,->] node [below] {$\quad \qquad\qquad (h,\alpha)\to (0,\infty)$}(m-1-2)
        
 		(m-2-1) edge node [below] {$h\to 0$} (m-2-2);

\end{tikzpicture}
\end{center}
\end{scriptsize}
\end{Remark}

\section{Boundary Optimization Problem with Variable \boldmath{$q$}}\label{sec4}

\subsection{Discrete Problem  ${(P^h_2)}$ Associated with $(P_2)$} 

Under the same considerations given in Section \ref{sec3.1} and taking into account Formula~\eqref{J2-J2alpha}, for a given $q \in Q$, we obtain the following quadratic cost function:
        \begin{equation}\label{J2}
        \begin{array}{ll}
        J_2(q)&=\tfrac{x_0 y_0}{2}\Big\{ q^2 x_0^2 \Big(\tfrac{1}{3} +\tfrac{M_2}{x_0^3}\Big)+q x_0 \Big(-\tfrac{5}{12} g x_0^2-  (b-z_d)\Big)  \\[0.5cm]
       &+\tfrac{2}{15} g^2 x_0^4 +  (b-z_d)^2+\tfrac{2}{3} g x_0^2 (b-z_d) \Big\}.
\end{array}       
       \end{equation}

  \noindent Then, the boundary optimal control of problem $(P_2)$, called $q_{op}$, and the associated continuous optimal state
  are given by:
 
      \begin{equation}\label{ec 4.2}
          q_{op}=\frac{\tfrac{5}{12} g x_0^2+ (b-z_d)}{2 x_0 \big(\tfrac{1}{3}+\tfrac{M_2}{x_0^3}\big)},\qquad u_{q_{op}}(x,y) = -\tfrac{1}{2} g\, x^2 + (g\, x_0 - q_{op}) x +b.
       \end{equation}


Associated with $(P_2)$, we define  the approximate discrete distributed optimal control problem ${(P^h_2)}$ on the constant heat flux $q$ as

$$ \text{ find }\quad {q^h_{op}}\in \mathbb{R}\quad\text{ such that }\quad
{J^h_2}({q^h_{op}})=\min\limits_{q\in \mathbb{R}} \text{ }{J^h_2}(q)
$$
where the  discrete cost function ${J^h_2}$ is defined by
    
     \[{J^h_2} (q)=  \tfrac{1}{2} \lVert {u^h_q} - z_d \rVert^2_H +  \tfrac{1}{2} M_2 \lVert q \rVert ^2_Q ,\]
     
 \noindent where ${u^h_q}$, given in (\ref{Def uh}), {denotes the discrete approximation for a fixed constant flux $q$}, $h$ is the spatial step, and $z_d$ (the desired state).
 From the definition of the norm over $Q$, it results~that:   
    \[{J^h_2} (q)= \tfrac{1}{2} y_0 \Big\{M_2\,q^2 \,x_0 + \, \sum_{i=1}^n \, \int_{x_i}^{x_{i+1}} [{u^h_q} (x,y) - z_d]^2 \,dx \Big \}\]
     
  \noindent and working algebraically, we get
  \begin{equation}\label{J2h(q)}
    \begin{array}{ll}
  &  {J^h_2}(q)=J_2(q)+\tfrac{x_0 y_0}{2}\Big\{  h \,g \,x_0 \Big[\tfrac{1}{3}q x_0 - \tfrac{5}{24} g x_0^2 - \tfrac{1}{2} (b-z_d)\Big] \\[0.45cm]
    & \qquad + h^2 g \Big[\tfrac{1}{36} x_0^2 g- \tfrac{1}{6}(b-z_d) + \tfrac{1}{12} q \,x_0\Big] + \tfrac{1}{24}  h^3g^2 x_0 + \tfrac{1}{180}  h^4  g^2\Big\}.    
    \end{array}    
  \end{equation}

\begin{Lemma}\label{Lema J2 J2h}
Given $q\in Q$ and $h>0$, we have
     \begin{equation}\label{ec 4.3} 
        \lvert  {J^h_2}(q) - J_2(q) \rvert  \approx C_7 \, h,
     \end{equation}
     where $C_7 = \tfrac{1}{2}x_0^2\, y_0\,|g| \Big \lvert  \tfrac{1}{3}q x_0-\tfrac{5}{24} gx_0^2-\tfrac{1}{2}(b-z_d) \Big \rvert$ is a constant independent of $h.$
      \end{Lemma}
      
\begin{proof}
It follows immediately from expression \eqref{J2h(q)} for ${J^h_2}$.
   \end{proof}

      
\begin{Lemma} Let us consider $h>0.$
    
     \begin{itemize}
      \item[(a)]  The explicit expression for the optimal variable ${q^h_{op}} $ is given by:     
     \begin{equation}\label{qhop}
      {q^h_{op}} =q_{op} - B_1 \,\tfrac{g h}{6} - B_2 \tfrac{g h^2}{24\,x_0} , \qquad 
     \end{equation}
     with   $B_1=B_2=\Big(\tfrac{1}{3} + \tfrac{M_2}{x_0^3}\Big)^{-1}.$
     
     \noindent 
     
      \item[(b)]The following error estimates hold:
     		\begin{equation}\label{qhop-qoo}
     		\lvert {q^h_{op}} - q_{op} \rvert \approx C_8 \, h  ,
     		\end{equation}
     	
     		\begin{equation}\label{J2hqhop-J2qop}
     	    \Big \lvert  {J^h_2}({q^h_{op}}) -  J_{2}(q_{op}) \Big \rvert \approx C_9 \, h,
     	    \end{equation}
     		
        where $C_8$  and $C_9$ are  constants independent of $h$.
     
     \end{itemize}
    \end{Lemma}

  \begin{proof}$ $
  \begin{itemize}
  \item[(a)] From the expression \eqref{J2h(q)} for ${J^h_2}$, we have
\begin{equation*}  
 { \frac{d}{dq} J^h_{2}(q)}= \tfrac{1}{2} x_0^2\, y_0 \Big\{ 2\,x_0 \,q \Big(\tfrac{1}{3} + \tfrac{M_2}{x_0^3} \Big) - \tfrac{5}{12} g \,x_0^2 - (b- z_d)  \\   
     +\tfrac{1}{3} \,g \,x_0\,  h+ \tfrac{1}{12} \,g \,h^2 \Big\}.
     \end{equation*} 
Therefore, $${q^h_{op}}= \frac{\tfrac{5}{12} g x_0^2+ (b-z_d)-\tfrac{1}{12} gh^2-\tfrac{1}{3} g x_0 h}{2 x_0 \big(\tfrac{1}{3}+\tfrac{M_2}{x_0^3}\big)}.$$     
   Taking into account that $q_{op}$ is given by \eqref{ec 4.2}, we obtain Formula \eqref{qhop}.
  \item[(b)]

On one hand, expression \eqref{qhop-qoo} is a direct  consequence of expression (\ref{qhop}) where \mbox{$C_8=\Big| \tfrac{g\,  B_1}{6} \Big|.$}
     
On the other hand, taking into account Formula  \eqref{J2h(q)}  for ${J^h_2}$, it follows that
   \begin{equation}\label{48}
   \begin{array}{ll}
  {J^h_2}({q^h_{op}})- J_2(q_{op}) =J_2({q^h_{op}})-J_2(q_{op})\\[0.45cm]
  + \tfrac{x_0^2y_0}{2} h \,g  \Big[\tfrac{1}{3}{q^h_{op}} x_0 - \tfrac{5}{24} g x_0^2 - \tfrac{1}{2} (b-z_d)\Big] +o(h^2).
  \end{array}
\end{equation}  
  From the definition of $J_2$ given by \eqref{J2} and the explicit expression \eqref{qhop} for ${q^h_{op}}$, we get
  \begin{equation}\label{49}
  \begin{array}{lll}
 &J_2({q^h_{op}})-J_2(q_{op})=\tfrac{x_0^2 y_0}{2}\Big\{ \tfrac{x_0}{B_1}({(q^h_{op})}^2-q_{op}^2)\\[0.45cm]
 & +({q^h_{op}}-q_{op})\Big(-\tfrac{5}{12} g x_0^2-  (b-z_d)\Big)\Big\}\\[0.45cm]
 &=\tfrac{x_0^2 y_0}{2} ({q^h_{op}}-q_{op})\Big( \tfrac{x_0}{B_1}({q^h_{op}}+q_{op}) -\tfrac{5}{12} g x_0^2-  (b-z_d)\Big)\\[0.45cm]
 &=\tfrac{x_0^2 y_0}{12}g h B_1 \Big(-\tfrac{2x_0}{B_1} q_{op} +\tfrac{5}{12} g x_0^2+  b-z_d\Big) +o(h^2).
 \end{array}
  \end{equation}
  

 \noindent Taking into account \eqref{48} and \eqref{49}, we obtain estimate \eqref{J2hqhop-J2qop}   with  $$C_9=\frac{|g| x_0^2 y_0 }{2} \Big\lvert\tfrac{B_1}{6}\Big(-\tfrac{2x_0}{B_1}q_{op}+\tfrac{5 }{12}g x_0^2 +b-z_d \Big)+ \tfrac{1}{3}q_{{op}} x_0 - \tfrac{5}{24} g x_0^2 - \tfrac{1}{2} (b-z_d)\Big\rvert. $$

  \end{itemize}
  \end{proof}

\begin{Lemma} Consider $u_{q_{op}}$, the solution of \eqref{ec 1.1} and \eqref{ec 1.2} for $q= q_{op}$, and ${u^h_{q^h_{op}}}$, 
 the discrete solution given by \eqref{Def uh} for each $h>0$ where $q={q^h_{op}}$ is the optimal variable of the problem ${(P^h_2)}$ given by \eqref{qhop}. Then, we have:
  \begin{equation}  \label{difu-problema2} {\rm (a)} \quad  			      
  				  \lVert {u^h_{q^h_{op}}}- u_{q_{op}}  \rVert _H \approx C_{10}\, h,
  \qquad \qquad
     {\rm (b)}\quad
        \left\lVert \pder[{u^h_{q^h_{op}}}]{x} - \pder[u_{q_{op}}]{x} \right\rVert _H \approx C_{11} \,h,
\end{equation}
 \noindent where $C_{10}$ and $C_{11}$ do not depend on the parameter $h.$
     
\end{Lemma}

\begin{proof}$ $
    \begin{itemize}
      
       \item[(a)] From the definition of $u$ and ${u^h}$ given by \eqref{Def: u y ualpha} and \eqref{Def uh}, respectively, it follows for $i= 1,\ldots, n$ that
  		      \begin{equation}
  		      \begin{array}{ll}
  		       {u^h_{q^h_{op}}}(x,y) - u_{q_{op}}(x,y) =  \tfrac{1}{2} \,g \,x^2-( {q^h_{op}} - q_{op} + g\,i\,h) x + g \,\tfrac{i^2-i}{2}\, h^2 , \quad \\[0.45cm]
  		        =\tfrac{1}{2} \,g \,x^2+h g x\left(\tfrac{B_1}{6}-i \right)+h^2g \left(\tfrac{B_1}{24}\tfrac{x}{x_0}+\tfrac{i(i-1)}{2} \right),\quad \forall x\in [x_i,\,x_{i+1}].
  		        \end{array}
  		        \end{equation}
  		      Then,  
\begin{equation*}
\begin{array}{ll} 
       \lVert {u^h_{q^h_{op}}} - u_{q_{op}}  \rVert _H^2= y_0 \displaystyle\sum\limits_{i=1}^n \displaystyle\int\limits_{x_i}^{x_{i+1}} \Big( {u^h_{q^h_{op}}}(x)-u_{q_{op}}(x)\Big)^2\; dx\\[0.45cm] 
      = x_0\, y_0\, g^2 \Big[\tfrac{1}{108} ( B_1-3)^2 h^2 \,x_0^2+\tfrac{1}{216} h^3\, x_0  (  B_1-3)^2 \\[0.45cm]
      +\frac{1}{8640} h^4\, (72 - 30 B_1 + 5 B_1^2)\Big].
           \end{array}
           \end{equation*}
           
Therefore, we get formula \eqref{difu-problema2} (a) with $C_{10}= \lvert  g (B_1-3)\rvert x_0 \sqrt{\tfrac{x_0 y_0}{108}}$ .
           
           \item [(b)] In the same manner, we get  
            $$\begin{array}{ll}
          &  \left\lVert \pder[{u^h_{q^h_{op}}}]{x} - \pder[u_{q_{op}}]{x}   \right\rVert _H^2 = y_0 \displaystyle \sum\limits_{i=1}^{n} \, \int\limits_{x_i} ^{x_{i+1}} \Big ( g\,x  -({q^h_{op}} - q_{op} ) - h\,g\,i\Big )^2\,dx \\
&=y_0\,g^2\Big[\tfrac{1}{36 }(12 - 6 B_1 + B_1^2) h^2   x_0 +\tfrac{1}{72} (B_1-3) B_1  h^3 +\tfrac{1}{576 }B_1^2  \frac{h^4}{x_0} \Big].
\end{array}$$
  	Then, we get \eqref{difu-problema2} (b) with $$C_{11}=\frac{|g|}{6}\sqrt{x_0 y_0 (12-6B_1+B_1^2)} .$$
 		         
    \end{itemize}
    \end{proof}
       
 \subsection{Discrete Problem  ${(P^h_{2\alpha})}$ Associated with $(P_{2\alpha})$} 
     
    If we suppose that the desired state $z_d$ is  constant in \eqref{J2-J2alpha}, the quadratic cost function  $J_{2\alpha}$ for optimal control problem $(P_{2\alpha})$ is explicitly given by:
 
      \begin{equation}\label{J2alpha-Explicito}
      \begin{array}{ll}
  J_{2\alpha}(q)=\dfrac{x_0 y_0}{2}\Big[ q^2 x_0^2 (D_{1\alpha} +\tfrac{M_2}{x_0^3})+q x_0 \Big(D_{2\alpha} g x_0^2+D_{3\alpha}  (b-z_d) \Big)\\[0.45cm]
   \qquad \qquad +D_{4\alpha} g^2 x_0^4 +D_{5\alpha}  (b-z_d)^2+D_{6\alpha} g x_0^2 (b-z_d) \Big],
  \end{array}
  \end{equation}
  
 \noindent where
\begin{equation}  
  \begin{array}{lll}
  D_{1\alpha}= \tfrac{1}{3}+\tfrac{1}{\alpha x_0}+\tfrac{1}{\alpha^2 x_0^2} ,  \quad & D_{2\alpha}= -\tfrac{5}{12}-\tfrac{5}{3\alpha x_0}-\tfrac{2}{\alpha^2 x_0^2}, \;\\
   D_{3\alpha}=- 1-\tfrac{2}{\alpha x_0},& D_{4\alpha}=\tfrac{2}{15} +\tfrac{2}{3\alpha x_0}+\tfrac{1}{\alpha^2 x_0^2} ,\\
    D_{5\alpha}=1,   & D_{6\alpha}=\tfrac{2}{3}+\tfrac{2}{\alpha x_0} .
  \end{array}
  \end{equation}

 \noindent Then, the continuous boundary optimization control, called $q_{\alpha_{op}}$,  and the associated state~are: 
\begin{equation} \label{qalphaop}
\begin{array}{ll}
q_{\alpha_{op}}=-\tfrac{D_{2\alpha} g x_0^2+D_{3\alpha} (b-z_d)}{2 x_0 \left(D_{1\alpha}+\tfrac{M_2}{x_0^3}\right)},\\ \\
 u_{\alpha q_{\alpha_{op}}}(x,y) = -\tfrac{1}{2} g\, x^2 + (g\, x_0 - q_{\alpha_{op}}) x + \tfrac{1}{\alpha} (g\,x_0 - q_{\alpha_{op}}) +b .
 \end{array}
 \end{equation} 
 
 \begin{Remark} Notice that $J_{2\alpha}(q)\to J_2(q) $ for all $q\in Q$ and $q_{\alpha_{op}}\to q_{op}$ when $\alpha \to \infty$. 
\end{Remark}

       \noindent Define the discrete cost function as:
\begin{equation}    
   {  J^h_{2 \alpha}} (q)=  \tfrac{1}{2} \lVert {u^h_{\alpha q}} - z_d \rVert^2_H +  \tfrac{1}{2} M_2 \lVert q \rVert ^2_Q ,
     \end{equation}
where {$u^h_{\alpha q}$} is the solution of ${(S^h_{\alpha})}$ given in (\ref{ec 2.3}) {when $q$ is fixed}. We set the following discrete optimization problem 
 ${(P^h_{2\alpha})}$ on the constant heat flux $q$ as

$$ \text{find }\quad { q^h_{\alpha_{op}}}\in \mathbb{R}\quad\text{ such that }\quad
{ J^h_{2 \alpha}}({ q^h_{\alpha_{op}}})=\min\limits_{q\in \mathbb{R}} \text{ }{ J^h_{2 \alpha}}({q}).
$$

  Working algebraically, the cost function ${ J^h_{2 \alpha}}$ can be written explicitly as:
  \begin{equation}\label{J2halpha-Explicito}
  \begin{array}{rl}
&  { J^h_{2 \alpha}}(q)= J_{2\alpha}(q) \\[0.45cm]
  &\quad+ \tfrac{1}{2} x_0\,y_0 \,g \,h \Big \{   \tfrac{ q\,x_0^2}{3}   - \tfrac{5\,g x_0^3}{24}  - \tfrac{(b-z_d)\,x_0}{2} 
    + \tfrac{3 x_0 q}{2\,\alpha} - \tfrac{7 g\,x_0^2}{6\alpha}
     - \tfrac{2(b-z_d)}{\alpha}+ \tfrac{2(q-g\,x_0)}{\alpha^2} \\[0.45cm]
    & \quad + h \Big( \tfrac{x_0^2\, g}{36}-  \tfrac{(b-z_d)}{6} + \tfrac{q \,x_0}{12} + \tfrac{2g\,x_0+q}{6\,\alpha}+\tfrac{g}{\alpha^2}\Big)\\[0.45cm]
    & \quad + h^2\,g \Big( \tfrac{x_0}{24} + \tfrac{1}{6\alpha} \Big) + \tfrac{1}{180} g\, h^3 \Big\}. 
    \end{array}
    \end{equation}

\begin{Lemma}\label{Lema 4.4} For each $q\in Q$  and $h>0$, we have:
\begin{equation}   \lvert  { J^h_{2 \alpha}}(q) - J_{2\alpha}(q) \rvert  \approx C_{7 \alpha} \, h, 
   \end{equation}
      with $$C_{7 \alpha} = \tfrac{1}{2}x_0^2\, y_0\, \Big \lvert g \,  \Big(  \tfrac{ q\,x_0}{3}   - \tfrac{5\,g x_0^2}{24}  - \tfrac{(b-z_d)}{2}+ \tfrac{3 q}{2\,\alpha} - \tfrac{7 g\,x_0}{6\alpha} - \tfrac{2(b-z_d)}{x_0 \,\alpha}+ \tfrac{2(q-g\,x_0)}{x_0\,\alpha^2}\Big)    \Big \rvert, $$ a constant independent of $h.$      
  
\end{Lemma}
\begin{proof}
It follows immediately from expression \eqref{J2halpha-Explicito}.
\end{proof}

\begin{Lemma} Let us consider $h>0.$
    
     \begin{itemize}
      \item[(a)]  The explicit expression for optimal control ${q^h_{\alpha_{op}}} $ is given by:     
     \begin{equation}\label{qhalphaop}
      { q^h_{\alpha_{op}}} = q_{\alpha_{op},}-B_{1\alpha}\frac{gh}{6}-B_{2\alpha}\frac{g h^2}{24 x_0} 
     \end{equation}
     with   $$\begin{array}{cc}
     B_{1\alpha}=\left(1+\tfrac{9}{2\alpha x_0}+\tfrac{6}{\alpha^2 x_0^2}\right) \left(D_{1\alpha}+\tfrac{M_2}{x_0^3} \right)^{-1},\\[0.45cm]
      B_{2\alpha}=\left(1-\tfrac{2}{\alpha x_0}\right) \left(D_{1\alpha}+\tfrac{M_2}{x_0^3} \right)^{-1}.\end{array}$$
     
     \noindent 
     
      \item[(b)]The following error estimates hold:
     		\begin{equation}\label{qhalphaop-qalphaop}
     		\lvert { q^h_{\alpha_{op}}} - q_{\alpha_{op}} \rvert \approx C_{8\alpha} \, h ,
     		\end{equation}
     		\begin{equation}\label{J2halpha qhalphaop-J2alpha qalphaop}
     	    \Big \lvert  { J^h_{2 \alpha}}({ q^h_{\alpha_{op}}}) -  J_{2\alpha}(q_{\alpha_{op}}) \Big \rvert \approx C_{9\alpha} \, h,
     	    \end{equation}

        where $C_{8\alpha}$  and $C_{9\alpha}$  do not depend on $h$.
     
     \end{itemize}
    \end{Lemma}
    \begin{proof} $ $

    \begin{itemize}
    \item[(a)]From the derivative of the control function ${ J^h_{2 \alpha}}$ given by
     \begin{equation*}
  \begin{array}{ll}
  { \frac{d}{dq} J^h_{2 \alpha}}(q)= {\frac{d}{dq}J_{2\alpha}}(q) + \tfrac{1}{2} x_0\,y_0 \,g \,h \Big \{   \tfrac{ x_0^2}{3} 
    + \tfrac{3 x_0 }{2\,\alpha} + \tfrac{2}{\alpha^2}
    + h \Big( \tfrac{x_0}{12} + \tfrac{1}{6\,\alpha} \Big) \Big\}\\[0.45cm] 
   \qquad \quad \ \ =\frac{x_0 y_0}{2} \Big\{2qx_0^2 \Big( D_{1\alpha}+\frac{M_2}{x_0^3}\Big) +x_0 \Big(D_{2\alpha} g x_0^2+ D_{3\alpha} (b-zd) \Big)\\[0.45cm]
   \qquad \qquad \; \; +gh\Big(   \tfrac{ x_0^2}{3} + \tfrac{3 x_0 }{2\,\alpha} + \tfrac{2}{\alpha^2}
    + h \Big( \tfrac{x_0}{12} + \tfrac{1}{6\,\alpha} \Big)\Big)\Big\},
    \end{array}
    \end{equation*}
    it follows that
    
      \begin{equation*}
  \begin{array}{ll}
  { q^h_{\alpha_{op}}}=  \frac{- gh\Big(   \tfrac{ x_0^2}{3} 
    + \tfrac{3 x_0 }{2\,\alpha} + \tfrac{2}{\alpha^2}\Big)-g h^2\Big( \tfrac{x_0}{12} + \tfrac{1}{6\,\alpha} \Big)- x_0 \Big(D_{2\alpha} g x_0^2+ D_{3\alpha} (b-zd) \Big)}{2x_0^2  \Big( D_{1\alpha}+\frac{M_2}{x_0^3}\Big)}.
    \end{array}
    \end{equation*}
    Working algebraically, we get formula \eqref{qhalphaop}.
    
    \item[(b)] Estimate  \eqref{qhalphaop-qalphaop} follows straightforwardly from \eqref{qhalphaop} with
    $$C_{8\alpha}=\left\lvert  \frac{g B_{1\alpha}}{6}\right\rvert.$$
    
From formula \eqref{J2halpha-Explicito} we obtain that
    \begin{equation}\label{resta1}
    \begin{array}{ll}
    { J^h_{2 \alpha}}({ q^h_{\alpha_{op}}})-J_{2\alpha}(q_{\alpha_{op}})=    J_{2\alpha}({ q^h_{\alpha_{op}}})-J_{2\alpha}(q_{\alpha_{op}}) \\[0.45cm]
    + \tfrac{1}{2} x_0\,y_0 \,g \,h \Big(   \tfrac{ { q^h_{\alpha_{op}}}\,x_0^2}{3}   - \tfrac{5\,g x_0^3}{24}  - \tfrac{(b-z_d)\,x_0}{2} 
    + \tfrac{3 x_0 { q^h_{\alpha_{op}}}}{2\,\alpha}\\[0.45cm]
    - \tfrac{7 g\,x_0^2}{6\alpha} - \tfrac{2(b-z_d)}{\alpha}+ \tfrac{2({ q^h_{\alpha_{op}}}-g\,x_0)}{\alpha^2}\Big)+o(h^2).
    \end{array}
    \end{equation}
Moreover, taking into account the explicit expression for ${J^h_{2\alpha}}$ given by  \eqref{J2halpha-Explicito} and \mbox{Formula~\eqref{qhalphaop}}, it follows that
\begin{equation}\label{resta2}
    \begin{array}{ll}
  {  J^h_{2\alpha}( q^h_{\alpha_{op}}})-J_{2\alpha}(q_{\alpha_{op}})= \frac{x_0^2 y_0}{2}\big({ q^h_{\alpha_{op}}}-q_{\alpha_{op}}\big) \Big[\big({ q^h_{\alpha_{op}}}+q_{\alpha_{op}}\big)x_0  \left(D_{1\alpha}+\tfrac{M_2}{x_0^3} \right)\\
    + D_{2\alpha}g x_0^2+D_{3\alpha}(b-z_d) \Big]\\
=-\frac{x_0^2 y_0}{12}  g B_{1\alpha}  \left(2q_1 x_0 \left(D_{1\alpha}+\frac{M_2}{x_0^3}\right)+D_{2\alpha} g x_0^2+D_{3\alpha} (b-z_d)\right)+o(h^2).
    \end{array}
\end{equation}    
    Combining \eqref{resta1} and \eqref{resta2} we get estimate \eqref{J2halpha qhalphaop-J2alpha qalphaop} with 
    \begin{equation*}
      \begin{array}{ll}
     C_{9\alpha}&=\frac{x_0^2 y_0}{2} |g| \,\Big| -\frac{B_{1\alpha}}{6}\Big(2 q_{\alpha_{op}} x_0 \Big(  D_{1\alpha}+\frac{M_2}{x_0^3}\Big)+D_{2\alpha}gx_0^2+D_{3\alpha}(b-z_d) \Big)\\
     \\
    & +  \frac{ q_{\alpha_{op}}\,x_0}{3}   - \frac{5\,g x_0^2}{24}  - \frac{(b-z_d)}{2} 
     + \frac{3  q_{\alpha_{op}}}{2\,\alpha} - \frac{7 g\,x_0}{6\alpha} - \frac{2(b-z_d)}{\alpha x_0}+ \frac{2(q_{\alpha_{op}}-g\,x_0)}{\alpha^2 x_0}\Big|.
     \end{array}
    \end{equation*}

    \end{itemize}
    \end{proof}

\begin{Lemma} Let us consider $u_{q_{\alpha_{op}}}$ the solution of \eqref{ec 1.1} and \eqref{ec 1.3} for $q= q_{\alpha_{op}}$ and ${u^h_{\alpha q^h_{\alpha_{op}}}}$ the discrete solution  given in (\ref{ec 2.3}) for  $h>0$ and $q={ q^h_{\alpha_{op}}}$. Then, we have:
						
	 \begin{equation}\begin{array}{ll} \label{difu-problema2alpha} {\rm (a)} \quad  			      
  				  \lVert {u^h_{\alpha q^h_{\alpha_{op}}}} - u_{\alpha q_{\alpha_{op}}}  \rVert _H \approx C_{10\,\alpha}\, h,
  \qquad \qquad \\
  \\
     {\rm (b)}\quad
        \left\lVert \pder[{u^h_{\alpha q^h_{\alpha_{op}}}}]{x} - \pder[u_{\alpha q_{\alpha_{op}}}]{x} \right\rVert _H \approx C_{11\, \alpha} \,h,
        \end{array}
\end{equation}
   \end{Lemma}
   
   \begin{proof} Similarly to what was done in Lemma \ref{Lema 4.4}, we obtain
   $$\begin{array}{ll}
   C_{10\, \alpha}=|g|x_0\\
    \sqrt{\tfrac{x_0 y_0}{108}} \sqrt{B_{1\alpha}^2 \left( \tfrac{3}{\alpha^2 x_0^2}+\tfrac{3}{\alpha x_0}+1\right)+9\left(\tfrac{12}{\alpha^2 x_0^2}+\tfrac{6}{\alpha x_0}+1 \right)-3B_{1\alpha}\left( \tfrac{12}{\alpha^2 x_0^2}+\tfrac{9}{\alpha x_0}+2\right)}\end{array}$$

      $$C_{11\, \alpha}=\frac{|g|}{6} \sqrt{ x_0 y_0 (12-6B_{1\alpha}+B_{1\alpha}^2)}.$$
   \end{proof}
   
\begin{Remark}
The constants verify that \mbox{$C_{i\alpha}\to C_i$}, when $ \alpha\to\infty, $ for each $ i=7,\dots,11.$
\end{Remark}

\begin{Remark}
The double convergence when $(h,\alpha)\to (0,+\infty)$ of the optimal control of problem ${(P^h_{2\alpha})}$ holds. The relationship among optimal control  problems  $(P_{2})$, $(P_{2\alpha})$, ${(P^h_2)}$ and ${(P^h_{2\alpha})}$ is given by the following diagram:
\begin{scriptsize}
\begin{center}
\begin{tikzpicture}
  \matrix (m) [matrix of math nodes,row sep=8em,column sep=8em,minimum width=2em]
  {
     \begin{array}{c} \text{Problem } {(P^h_2)}\\ {q^h_{op}}, {J^h_2}({q^h_{op}}),  u_{h {q^h_{op}}} \end{array}& \begin{array}{c} \text{Problem } (P_{2})\\ q_{op}, J_{q}(q_{op}),  u_{q_{op}} \end{array} \\
   \begin{array}{c} \text{Problem } {(P^h_{2\alpha})}\\ { q^h_{\alpha_{op}}}, { J^h_{2 \alpha}}({ q^h_{\alpha_{op}}}),  {u^h_{\alpha q^h_{\alpha_{op}}}} \end{array}& \begin{array}{c} \text{Problem } (P_{2\alpha})\\ q_{\alpha_{op}}, J_{2\alpha}(q_{\alpha_{op}}),  u_{\alpha q_{\alpha_{op}}} \end{array}\\};
  \path[-stealth]
      	 (m-2-1) edge node [left] {$\alpha\to \infty$} (m-1-1)
       
        (m-2-2) edge node [right] {$\alpha\to \infty$} (m-1-2)
        
        (m-1-1) edge node [below] {$h\to 0$} (m-1-2)
        
          (m-2-1) edge [dashed,->] node [below] {$\quad \qquad\qquad (h,\alpha)\to (0,\infty)$}(m-1-2)
        
 		(m-2-1) edge node [below] {$h\to 0$} (m-2-2);

\end{tikzpicture}
\end{center}
\end{scriptsize}

\end{Remark}

 \vspace{.2cm}

\section{Boundary Optimization Problem with Variable \boldmath{$b$}}\label{sec5}

\subsection{Discrete Problem  ${(P^h_{3})}$ Associated with $(P_3)$} 
In this section we consider the boundary optimal control problem  $(P_3)$ given by \eqref{PControlJ3}. 
Taking into account expression \eqref{J3-J3alpha}, for a given constant $b$, we get
    
    \begin{equation}\label{J3-explicito}
    \begin{array}{ll}
     J_3(b)=\tfrac{x_0\,y_0}{2}\Big\{ b^2 \Big(1+ \tfrac{M_3}{x_0} \Big) + b   \Big( \tfrac{2\,g\,x_0^2 }{3}- q\,x_0 - 2\,z_d \Big) \\[0.45cm]
     \hspace{1.4cm}+\Big [\tfrac{2g^2 x_0^4}{15} -\tfrac{5 g q x_0^3}{12} + \tfrac{x_0^2}{3} (q^2-2 g z_d) + z_d(z_d + q x_0) \Big]\Big\}.
     \end{array} 
    \end{equation}

    \noindent Then, the boundary optimal variable of problem $(P_3)$, called $b_{op}$, and the associated continuous optimal state, are given respectively by:
    
    \begin{equation}\label{bop-ubop}
     b_{op}= \frac{-\tfrac{g x_0^2}{3} + \tfrac{q x_0}{2} + z_d}{1+ \tfrac{M}{x_0}}, \qquad 
     u_{b_{op}}=  -\tfrac{1}{2} g\, x^2 + (g\, x_0 - q) x +b_{op}.
    \end{equation}

We define the discrete optimal control problem ${(P^h_{3})}$ on the constant temperature $b$ as

$$ \text{ find }\quad {b^h_{op}}\in \mathbb{R}\quad\text{ such that }\quad
{J^h_{3}}({b^h_{op}})=\min\limits_{b\in \mathbb{R}} \text{ }{J^h_{3}}(b)
$$
where the discrete cost function  ${J^h_{3}}(b)$ is defined as:
    
     \[{J^h_{3}} (b)=  \tfrac{1}{2} \lVert {u^h_b} - z_d \rVert^2_H +  \tfrac{1}{2} M_3 \lVert b \rVert ^2_B, \]
     
 \noindent where ${u^h_b}$ is given in (\ref{Def uh}) {for a fixed constant $b$}, $h$ is the spatial step, and $z_d$ (the desired state) is constant.

Notice that the cost function ${J^h_{3}}$ can be explicitly written as:
   \begin{equation}\label{J3h-explicito}
   \begin{array}{ll}
{J^h_{3}} (b)= J_3(b) +\tfrac{x_0 y_0 g}{2} \Big\{ - b h  \Big( \tfrac{x_0}{2} + \tfrac{h}{6}  \Big) +h  x_0 \big(\tfrac{1}{3} q x_0-\tfrac{5}{24} g x_0^2+\tfrac{z_d}{2} \big)\\[0.45cm]
\hspace{1.45cm}+\tfrac{1}{6} h^2 \Big(\tfrac{q x_0}{2}+\tfrac{1}{6} g  x_0^2+z_d \Big) +\tfrac{1}{24} g h^3 x_0+\tfrac{1}{180}g h^4\Big\}.
\end{array}
\end{equation}

\begin{Lemma}\label{Lema J3 J3h}
Let  $b \in R$ and $h>0$; we have:
    \begin{equation}\label{ec 5.3} 
        \lvert  {J^h_{3}}(b) - J_3(b) \rvert  \approx C_{12} \, h, 
     \end{equation}
     where  $$C_{12} = \tfrac{1}{2}x_0^2\, y_0\,| g| \Big \lvert  -\tfrac{b}{2} + \tfrac{q x_0}{3} -\tfrac{5 g x_0^2}{24} + \tfrac{z_d}{2} \Big \rvert,$$ 
   does not depend on $h.$
      
      \end{Lemma}
\begin{proof}
It follows from expression \eqref{J3h-explicito} for ${J^h_{3}}$.
\end{proof}    \noindent

\begin{Lemma} Let us consider $h > 0$.
    
     \begin{itemize}
      \item[(a)]  The explicit expression for the optimal variable ${b^h_{op}} $ is given by:     
     \begin{equation}\label{bhop}
      {b^h_{op}} =b_{op} + E_1 g x_0 \,h \Big(1+ \tfrac{h}{3 x_0}\Big), \qquad E_1 = \frac{1}{4 \Big(1+ \tfrac{M_3}{x_0} \Big)  }.
     \end{equation}
     
     \noindent 
     
      \item[(b)] The following error estimates hold:
  
     		\begin{equation}\label{bhop-bop}  
     				      \lvert {b^h_{op}} - b_{op} \rvert \approx C_{13} \, h ,
     				      \end{equation}\label{J3hbop-J3bop}
     				      \begin{equation}     		\quad 		     \Big \lvert  {J^h_{3}}({b^h_{op}}) -  J_3(b_{op}) \Big \rvert \approx C_{14} \, h,
     				      \end{equation}
    		
        where $C_{13}$  and $C_{14}$ do not depend on $h$.
     
     \end{itemize}
    
  \end{Lemma}
\begin{proof} $ $

\begin{itemize}
\item[(a)] According to \eqref{J3h-explicito} we have

  \begin{equation}\label{J3h-derivada}
{ \frac{d}{db}J^h_{3}} (b)= \frac{d}{db} J_3(b) -\tfrac{x_0 y_0 }{2} g  b h  \Big( \tfrac{x_0}{2} + \tfrac{h}{6}  \Big) .
\end{equation}
Then, formula \eqref{bhop} for ${b^h_{op}}$ follows immediately.

\item[(b)] Estimate \eqref{bhop-bop}  is a direct consequence of \eqref{bhop} with $C_{13}=|E_1 g x_0|$. 

Moreover, taking into account Formulas \eqref{J3-explicito} and \eqref{J3h-explicito} for $J_3$ and ${J^h_{3}}$ and Formulas \eqref{bop-ubop} and \eqref{bhop} for $b_{op}$ and ${b^h_{op}}$, respectively, it follows that
$$\begin{array}{l}
{J^h_{3}} ({b^h_{op}})- J_3(b_{op})\\[0.45cm]
=J_{3} ({b^h_{op}})- J_3(b_{op}) +\tfrac{x_0^2 y_0 g}{2} \Big\{ - \tfrac{{b^h_{op}}}{2} + \tfrac{1}{3} q x_0-\tfrac{5}{24} g x_0^2+\tfrac{z_d}{2} \Big\} h+ o(h^2).
\end{array}
$$
In addition, the expression  $J_{3} ({b^h_{op}})- J_3(b_{op})$ can be rewritten as
$$
J_{3} ({b^h_{op}})- J_3(b_{op})=\tfrac{x_0^2 y_0 g }{2} E_1   \Big( 2b_{op} \Big(1+ \tfrac{M_3}{x_0} \Big)+\tfrac{2}{3} g x_0^2-q x_0-2 z_d\Big)h+o(h^2).
$$
Therefore, it follows that estimate \eqref{J3hbop-J3bop} is given by  
 \[ C_{14}=  \tfrac{x_0^2  y_0 g}{2} \Big\lvert E_1  \Big( 2b_{op} \Big(1+ \tfrac{M_3}{x_0} \Big)+\tfrac{2}{3} g x_0^2-q x_0-2 z_d\Big)-\tfrac{b_{op}}{2} + \tfrac{q x_0}{3} - \tfrac{5 g x_0^2}{24}+ \tfrac{z_d}{2}\Big \rvert  \textbf{.}\]
\end{itemize}
\end{proof}

\begin{Lemma} Let us consider $u_{b_{op}}$, the solution of \eqref{ec 1.1}, \eqref{ec 1.3} for $b= b_{op}$, and
		${ u^h_{b^h_{op}}}$, the discrete solution  given in (\ref{ec 2.3}) for  $h>0$ and $b={b^h_{op}}.$ Then, we have:
						
	 \begin{equation}  \label{difu-problema3} {\rm (a)} \quad  			      
  				  \lVert {u^h_{b^h_{op}}} - u_{b_{op}}  \rVert _H \approx C_{15}\, h,
  \quad \quad
     {\rm (b)}\quad
        \left\lVert \pder[{u^h_{b^h_{op}}}]{x} - \pder[u_{b_{op}}]{x} \right\rVert _H \approx C_{16} \,h,
\end{equation}
where $C_{15}$ and $C_{16}$ are constants that do not depend on $h$.
   \end{Lemma}   
    
 \begin{proof}
 Working algebraically, we obtain 
$$  \lVert {u^h_{b^h_{op}}} - u_{b_{op}}  \rVert^2 _H=\tfrac{x_0^3  y_0 g^2}{2} \Big(2E_1^2-E_1+\tfrac{1}{6} \Big) h^2+o(h^3). $$
Then, we obtain estimate $\rm (a)$ with
 $$C_{15}=x_0  |g| \sqrt{\tfrac{x_0 y_0}{2}}\sqrt{2E_1^2-E_1+\tfrac{1}{6} }.$$
 In a similar manner, we get that estimate $\rm{b})$ holds with
 $$C_{16}=|g|\sqrt{\tfrac{x_0 y_0}{2}}.$$
 \end{proof}
 
   \subsection{Discrete Problem  ${(P^h_{3\alpha})}$ Associated with $(P_{3\alpha})$} 

   From~\cite{BoGaTa2020}, we know that  the continuous quadratic  functional cost in \eqref{J1-J1alpha} for the optimization problem  $(P_{3\alpha})$  is explicitly given by: 
      \begin{equation}\label{J3alpha-explicito}
          J_{3\alpha}(b)= J_3(b) +\tfrac{x_0 y_0}{2}\Big\{ \tfrac{1}{3\alpha}  (q - g x_0) (-6b + 3 q x_0 - 2 g x_0^2 + 6 z_d) +  \tfrac{1}{\alpha^2} (  q- g x_0)^2 \Big\}
      \end{equation}
      where $J_3$ is defined by \eqref{J3-explicito}.
      Moreover, the continuous optimal boundary control $b_{\alpha_{op}}$ is given by
\begin{equation}\label{balphaop}
          b_{\alpha_{op}} =b_{op}-  \frac{ g x_0-q}{\alpha (1+ \tfrac{M_3}{x_0})} .
     \end{equation}
\noindent The continuous associated state is established by: 
     \begin{equation}\label{ualpha-balphaop}
\qquad u_{b_{\alpha_{op}}}(x,y) = -\tfrac{1}{2} g\, x^2 + (g\, x_0 - q) x + \tfrac{1}{\alpha} (g\,x_0 - q) +b_{\alpha_{op}} .
       \end{equation}

\vspace{.2cm}
    
\noindent Define the discrete cost function as:
\begin{equation}    
    { J^h_{3\alpha}} (b)=  \tfrac{1}{2} \lVert {u^h_{\alpha b}} - z_d \rVert^2_H +  \tfrac{1}{2} M_3 \lVert b \rVert ^2_B 
     \end{equation}
where ${u^h_{\alpha b}}$ is the solution of ${ (S^h_{\alpha })}$ given in (\ref{ec 2.3}) {for a fixed $b$}.  We set the following discrete optimization problem ${(P^h_{3\alpha})}$ as

$$ \text{find }\quad {b^h_{\alpha_{op}}}\in \mathbb{R}\quad\text{ such that }\quad
{J^h_{3\alpha}}({b^h_{\alpha_{op}}})=\min\limits_{b\in \mathbb{R}} \text{ }{J^h_{3\alpha}}(b).
$$

Working algebraically leads us to write ${J^h_{3\alpha}}$ as follows:
  \begin{equation}\label{J3halpha-Explicito}
  \begin{array}{ll}
 { J^h_{3\alpha}}(b)= J_{3\alpha}(b) + \tfrac{1}{2} x_0\,y_0 \,g \,h \Big \{- b \Big(  \tfrac{x_0}{2} + \tfrac{2}{\alpha}  + \tfrac{h}{6}\Big)
  + g \Big( \tfrac{-5 x_0^3}{24} -  \tfrac{7 x_0^2}{6 \alpha} -  \tfrac{2 x_0}{\alpha^2} \Big)\\[0.45cm]
  + q \Big(  \tfrac{x_0^2}{3} +  \tfrac{3 x_0}{2 \alpha} +  \tfrac{2}{\alpha^2}\Big) + z_d \Big(  \tfrac{x_0}{2} +  \tfrac{2}{\alpha}\Big) \\[0.45cm]
  + h \Big[ g \Big(  \tfrac{x_0^2}{36} +  \tfrac{x_0}{3 \alpha} +  \tfrac{1}{\alpha^2}\Big) + q \Big(  \tfrac{x_0}{12} + \tfrac{1}{6 \alpha}\Big) + \tfrac{z_d}{6} \Big]
 + h^2 g \Big(  \tfrac{x_0}{24} +  \tfrac{1}{6 \alpha}\Big) + h^3 \tfrac{g}{180} \Big\}.
 \end{array}
 \end{equation}
        
  \vspace{.4cm}

  \begin{Lemma} For $b\in B$ and $h>0$, we have
\begin{equation}    \label{J3halpha-J3alpha}
 \lvert  { J^h_{3\alpha}}(b) - J_{3\alpha}(b) \rvert  \approx C_{12\alpha} \, h ,
\end{equation}
with $$C_{12\alpha}=\tfrac{x_0\,y_0}{2}  |g|  \Big\lvert- b \Big(  \tfrac{x_0}{2} + \tfrac{2}{\alpha}  + \tfrac{h}{6}\Big)
  + g \Big( \tfrac{-5 x_0^3}{24} -  \tfrac{7 x_0^2}{6 \alpha} -  \tfrac{2 x_0}{\alpha^2} \Big)
  + q \Big(  \tfrac{x_0^2}{3} +  \tfrac{3 x_0}{2 \alpha} +  \tfrac{2}{\alpha^2}\Big) + z_d \Big(  \tfrac{x_0}{2} +  \tfrac{2}{\alpha}\Big)\Big\rvert.$$
  \end{Lemma}
  \begin{proof}
  It arises immediately from \eqref{J3halpha-Explicito}.
  \end{proof}
  
\begin{Lemma} Let us consider $h > 0$.
    
     \begin{itemize}
      \item[(a)]  The explicit expression for optimal control ${b^h_{\alpha_{op}}} $ is given by:     
     \begin{equation}\label{bhalphaop}
      {b^h_{\alpha_{op}}} =b_{\alpha_{op}} + E_1 g x_0 \,h \Big(1+\tfrac{4}{\alpha x_0}+ \tfrac{h}{3 x_0}\Big),
     \end{equation}
     where $E_1$ is given in \eqref{bhop}.
     \noindent 
     
      \item[(b)] The following error estimates hold:
  
     		\begin{equation}\label{bhalphaop-balphaop}  
     				      \lvert {b^h_{\alpha_{op}}} - b_{\alpha_{op}} \rvert \approx C_{13\alpha} \, h ,
     				      \end{equation}\label{J3halphabhalphaop-J3alphabalphaop}
     				      \begin{equation}     		\quad 		     \Big \lvert { J^h_{3\alpha}}({b^h_{\alpha_{op}}}) -  J_{3\alpha}(b_{\alpha_{op}}) \Big \rvert \approx C_{14\alpha} \, h,
     				      \end{equation}
    		
        where $C_{13\alpha}$  and $C_{14\alpha}$ do not depend on $h$.
     
     \end{itemize}
    
  \end{Lemma}
\begin{proof} $ $
\begin{itemize}
\item[(a)] It follows from expression ${J^h_{3\alpha}}$ given by \eqref{J3halpha-Explicito}.

\item[(b)] The estimate in \eqref{bhalphaop-balphaop} is obtained immediately from item $\rm{a)}$ with $$C_{13\alpha}=|E_1| |g| x_0  \Big\lvert 1+\tfrac{4}{\alpha x_0}\Big\rvert.$$

Taking into account \eqref{J3halpha-Explicito} and \eqref{J3alpha-explicito} yields 
$${J^h_{3\alpha}}({b^h_{\alpha_{op}}})-J_{3\alpha}(b_{\alpha_{op}})=J_{3\alpha}({b^h_{\alpha_{op}}})-J_{3\alpha}(b_{\alpha_{op}})+F_{1\alpha}h+o(h^2).$$
with $$F_{1\alpha}=\tfrac{1}{2} x_0^2\,y_0 \,g\Big(- (b_{\alpha_{op}}-z_d) \Big(  \tfrac{1}{2} + \tfrac{2}{\alpha x_0}  \Big)
  + g x_0^2 \Big( \tfrac{-5 }{24} -  \tfrac{7 }{6 \alpha x_0} -  \tfrac{2 }{\alpha^2 x_0^2} \Big)
  + q x_0\Big(  \tfrac{1}{3} +  \tfrac{3 }{2 \alpha x_0} +  \tfrac{2}{\alpha^2 x_0^2}\Big) \Big).$$

In addition, from  the definition of ${ J^h_{3\alpha}}$ and ${b^h_{\alpha_{op}}}$, we have
$${ J^h_{3\alpha}}({b^h_{\alpha_{op}}})-J_{3\alpha}(b_{\alpha_{op}})=J_3({b^h_{\alpha_{op}}})-J_3(b_{\alpha_{op}})+F_{2\alpha}h+o(h^2)$$
with $$ F_{2\alpha}=\tfrac{x_0^2 y_0 g}{\alpha} E_1 (-q+gx_0)\left( 1+\tfrac{4}{\alpha x_0}\right).$$

Finally, according to formula \eqref{J3-explicito} for $J_3$, we get 
$$J_3({b^h_{\alpha_{op}}})-J_3(b_{\alpha_{op}})=F_{3\alpha}h+o(h^2),$$
with  $$F_{3\alpha}=\tfrac{x_0^2 y_0 g }{2} E_1 \Big( 1+\tfrac{4}{\alpha x_0}\Big)  \Big( 2b_{\alpha_{op}} \Big(1+ \tfrac{M_3}{x_0} \Big)+\tfrac{2}{3} g x_0^2-q x_0-2 z_d\Big).$$

Therefore,  estimate \eqref{J3halphabhalphaop-J3alphabalphaop} holds for 
$$
C_{14\alpha}=\Big\lvert F_{1\alpha}+F_{2\alpha}+F_{3\alpha}\Big\rvert. $$
\end{itemize}
\end{proof}

\begin{Lemma} Let us consider $u_{b_{\alpha_{op}}}$, the solution of \eqref{ec 1.1} and \eqref{ec 1.3} for $b= b_{\alpha_{op}}$, and
		${u^h_{\alpha b^h_{\alpha_{op}}}}$, the discrete solution  given in (\ref{ec 2.3}) for  $h>0$ and $b={b^h_{\alpha_{op}}}.$ Then, we have:
						
	 \begin{equation}  \label{difu-problema3alpha} {\rm (a)} \quad  			      
  				  \lVert {u^h_{\alpha b^h_{\alpha_{op}}}} - u_{\alpha b_{\alpha_{op}}}  \rVert _H \approx C_{15\,\alpha}\, h,
 \quad
     {\rm (b)}\quad
        \left\lVert \pder[{u^h_{\alpha b^h_{\alpha_{op}}}}]{x} - \pder[u_{\alpha b_{\alpha_{op}}}]{x} \right\rVert _H \approx C_{16\, \alpha} \,h,
\end{equation}
   \end{Lemma}
   
   \begin{proof} Similarly to what was done in Lemma \ref{Lema 4.4}, we obtain
   $$\begin{array}{c}
   C_{15\, \alpha}=x_0  |g|  \sqrt{\tfrac{x_0 y_0}{2}} \mathcal{A}, \\[0.45cm]
  \mathcal{A}= \sqrt{E_1^2\left(2+\tfrac{16}{\alpha x_0}+\tfrac{32}{\alpha^2x_0^2} \right)+E_1\left(-1-\tfrac{8}{\alpha x_0}-\tfrac{16}{\alpha^2 x_0^2} \right)+\tfrac{1}{6}+\tfrac{1}{\alpha x_0}+\tfrac{2}{\alpha^2 x_0^2} },\\[0.45cm]
  C_{16\, \alpha}=C_{16}.
   \end{array}$$
   \end{proof}
\begin{Remark}
The constants obtained in the estimates of the previous lemmas verify that \mbox{$C_{i\alpha}\to C_i$}  when $ \alpha\to\infty$ for $i=12,\cdots,16$.
\end{Remark}

\begin{Remark}
The double convergence when $(h,\alpha)\to (0,+\infty)$ of the optimal control of problem ${(P^h_{3\alpha})}$ holds. The relationship among the optimal control of problems  $(P_{3})$, $(P_{3\alpha})$, ${(P^h_{3})}$ and ${(P^h_{3\alpha})}$  is given by the following diagram:
\begin{scriptsize}
\begin{center}
\begin{tikzpicture}
  \matrix (m) [matrix of math nodes,row sep=8em,column sep=8em,minimum width=2em]
  {
     \begin{array}{c} \text{Problem } {(P^h_{3})}\\ {b^h_{op}}, {J^h_{3}}({b^h_{op}}),  {u^h_{b^h_{op}}} \end{array}& \begin{array}{c} \text{Problem } (P_{3})\\ b_{op}, J_{b}(b_{op}),  u_{b_{op}} \end{array} \\
   \begin{array}{c} \text{Problem } {(P^h_{3\alpha})}\\ {b^h_{\alpha_{op}}},{ J^h_{3\alpha}}({b^h_{\alpha_{op}}}),  {u^h_{\alpha b^h_{\alpha_{op}}}}\end{array}& \begin{array}{c} \text{Problem } (P_{3\alpha})\\ b_{\alpha_{op}}, J_{3\alpha}(b_{\alpha_{op}}),  u_{\alpha b_{\alpha_{op}}} \end{array}\\};
  \path[-stealth]
      	 (m-2-1) edge node [left] {$\alpha\to \infty$} (m-1-1)
       
        (m-2-2) edge node [right] {$\alpha\to \infty$} (m-1-2)
        
        (m-1-1) edge node [below] {$h\to 0$} (m-1-2)
        
          (m-2-1) edge [dashed,->] node [below] {$\quad \qquad\qquad (h,\alpha)\to (0,\infty)$}(m-1-2)
        
 		(m-2-1) edge node [below] {$h\to 0$} (m-2-2);

\end{tikzpicture}
\end{center}
\end{scriptsize}

\end{Remark}

\section{Numerical Results}\label{sec6}
  We carried out some numerical simulations in order to illustrate  the theoretical results obtained in the previous sections for the optimal control problems ${(P^h_{i})}$ and ${(P^h_{i\alpha})}$ for $i=1,2,3$.

  Throughout this section we consider the domain $\Omega=[0,1]\times [0,1]$, i.e, $x_0=y_0=1$.

Before analyzing the optimal control problems we illustrate the behavior of the continuous state of the systems $(S)$ and $(S_\alpha)$ and the discrete state of the systems $(S_{h})$ and~$({S^h_\alpha})$.

In Figure \ref{fig1} (a) we plotted the state of system $u$ given by \eqref{Def: u y ualpha} and the approximate  discrete function ${u^h}$ defined by \eqref{Def uh} against the position $x$ for  $h=1/3, 1/5, 1/10$. As we saw in Lemma \ref{Cota u uh} for each fixed $x$, the functions ${u^h}(x)$ increase and get closer to the limit $u(x)$ as $h$ decreases. 
In a similar manner, in Figure \ref{fig1} (b), for $\alpha=50$,  we obtained system $u_\alpha$ given by \eqref{Def: u y ualpha} and the approximate  discrete function ${u^h_\alpha}$ defined by \eqref{ec 2.3} against the position $x$ for  $h=1/3, 1/5, 1/10$.
Notice that as  $h$ decreases, the functions $\lbrace{{u^h_\alpha}\rbrace}$ increase and get closer to the limit $u_\alpha$ as it was proved in Lemma \ref{lema convergencia uhalpha a ualpha}.
\vspace{-3pt}

  \begin{figure}[h]
\centering
\begin{tabular}{cc}
\includegraphics[scale=0.5]{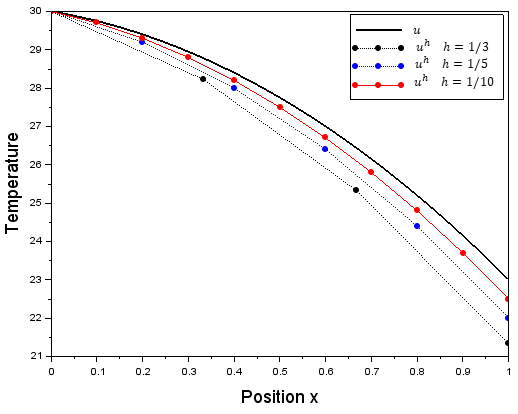} &
\includegraphics[scale=0.5]{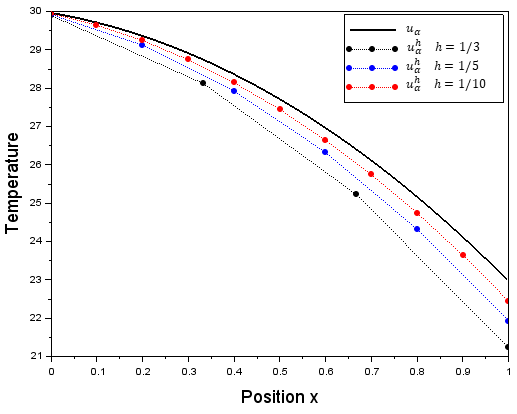} \\
{\scriptsize (a) $u$ and ${u^h}$ for different values of $h$} &
{\scriptsize (b) $u_\alpha$ and ${u^h_\alpha}$ for different values of $h$ with $\alpha=50$}
\end{tabular}

\caption{State of systems $(S)$, $({S^h})$, $(S_\alpha)$ and $({S^h_\alpha})$ using $q=12$, $b=30$, $z_d=40$ and $g=10$.}
\label{fig1}
\end{figure}

  In addition in order to visualize the double convergence of ${u^h_\alpha}\to u$ when \mbox{$(h,\alpha)\to (0,\infty)$,} in Figure \ref{uuhalpha} we plotted $u$ and ${u^h_\alpha}$ for $(h,\alpha)=(1/3,10), (1/5,50)$ and $(1/10,500)$.

 \begin{figure}[h!!]
 \centering
 \includegraphics[scale=0.6]{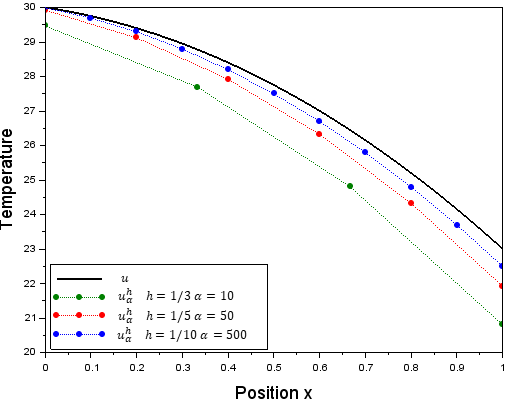}
\caption{Plot of $u$ and ${u^h_\alpha}$ against $n=1/h$ for different values of $(h,\alpha)$.}\label{uuhalpha}
 \end{figure}

 {Table \ref{tab:L2errors} illustrates that the $L^2$ errors exhibit a linear rate of convergence. 
Indeed, each refinement step in which the mesh size $h$ is divided by two produces an error that is approximately halved, confirming the expected first-order behavior.}
\begin{table}[h!!]
\caption{$L^2$ errors for $u-u^h$ and $u_{\alpha}-u_{\alpha}^h$ for different values of $\alpha$.}
\label{tab:L2errors}
\centering
\begin{tabularx}{\textwidth}{c cccc}
\toprule
\boldmath{$h$} &\boldmath{$\|u-u^h\|_{L^2}$}
& \multicolumn{3}{c}{\boldmath{$\|u_{\alpha}-u_{\alpha}^h\|_{L^2}$}} \\
\midrule
 &  & $\alpha=50$ 
      & $\alpha=100$ 
      & $\alpha=200$ \\
\midrule
0.250000 & 7.675914 × 10\textsuperscript{$-$1}
 & 8.120374 × 10\textsuperscript{$-$1} & 7.897314 × 10\textsuperscript{$-$1} & 7.786397 × 10\textsuperscript{$-$1} \\
\midrule
0.125000 & 3.722243 × 10\textsuperscript{$-$1} & 3.942740 × 10\textsuperscript{$-$1} & 3.832039 × 10\textsuperscript{$-$1} & 3.777023 × 10\textsuperscript{$-$1} \\
\midrule
0.062500 & 1.832549 × 10\textsuperscript{$-$1} & 1.942324 × 10\textsuperscript{$-$1} & 1.887200 × 10\textsuperscript{$-$1} & 1.859813 × 10\textsuperscript{$-$1} \\
\midrule
0.031250 & 9.091783 × 10\textsuperscript{$-$2} & 9.639423 × 10\textsuperscript{$-$2} & 9.364392 × 10\textsuperscript{$-$2} & 9.227771 × 10\textsuperscript{$-$2} \\
\midrule
0.015625 & 4.528211 × 10\textsuperscript{$-$2} & 4.801716 × 10\textsuperscript{$-$2} & 4.664351 × 10\textsuperscript{$-$2} & 4.596121 × 10\textsuperscript{$-$2} \\
\bottomrule
\end{tabularx}

\end{table}

 \subsection{Control Variable $g$}

   In this subsection we obtain some computational examples for the optimal distributed control problems $(P_1)$, ${(P^h_1)}$, $(P_{1\alpha})$ and ${(P^h_{1\alpha})}$. For each plot, we set $q=12, b=30, z_d=40$ and $M_1=1$.

    In Figure \ref{figJ1J1h} we plotted the continuous quadratic cost function $J_1$ given by \eqref{J1-explicito}  and the discrete cost function ${J^h_1}$ obtained in \eqref{J1h-explicito} against $g$ for  $h=1/10$, $1/50$ and $1/100$. Notice that as $h$ decreases, the function ${J^h_1}={J^h_1}(g)$  also decreases to the limit function $J_1=J_1(g)$ in agreement with Lemma \ref{Lema:J1(g)-J1h(g)}. In a similar manner in Figure \ref{FigJ1alphaJ1halpha}, for $\alpha=50$, we obtain the continuous function $J_{1\alpha}$ and the discrete functions ${J^h_{1\alpha}}$ for $h=1/10,1/50$ and $1/100$ observing the convergence of ${J^h_{1\alpha}}\to J_{1\alpha}$ as $h$ decreases to zero.
     Moreover, Figure \ref{figJ1double} shows the double convergence of ${J^h_{1\alpha}}\to J_1$ when $(h,\alpha)\to (0,\infty)$. We illustrate how ${J^h_{1\alpha}}$ gets closer to $J_1$ as the value of $h$ decreases and the value of $\alpha$ increases.

\begin{figure}[h!]
\centering
\includegraphics[scale=0.5]{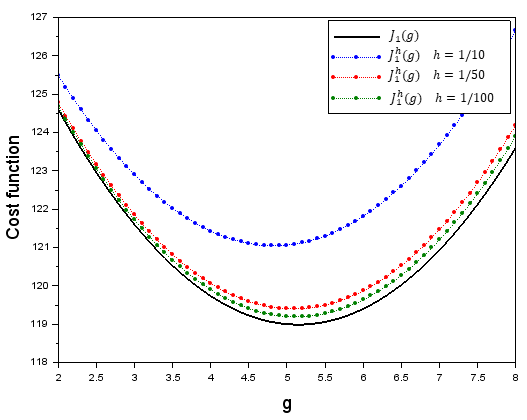}
\caption{Plot of $J_1$ and $J^h_1$ against $g$.}
\label{figJ1J1h}
\end{figure}

\begin{figure}[h!]
\centering
\includegraphics[scale=0.5]{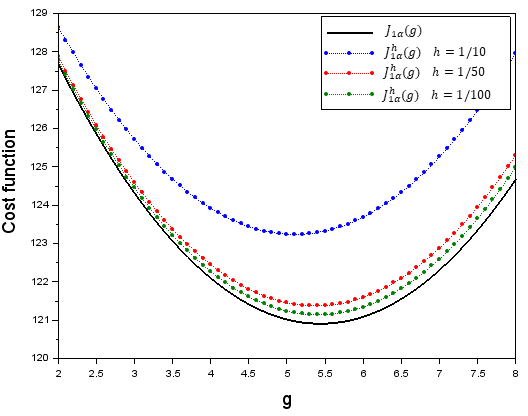}
\caption{Plot of $J_{1\alpha}$ and $J^h_{1\alpha}$ for $\alpha=50$ against $g$.}
\label{FigJ1alphaJ1halpha}
\end{figure}

\begin{figure}[h!]
\centering
\includegraphics[scale=0.5]{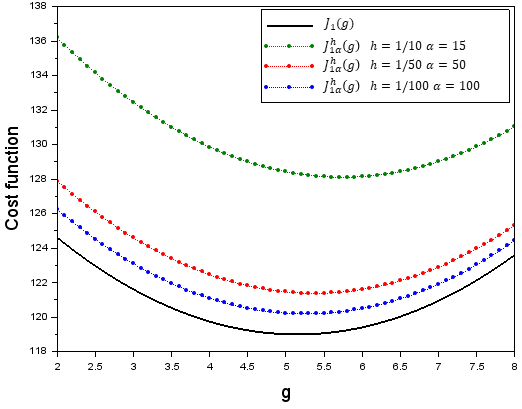}
\caption{Plot of $J_1$ and $J^h_{1\alpha}$ against $g$.}
\label{figJ1double}
\end{figure}

In Figure \ref{figgop} we plotted the continuous optimal control $g_{op}$  for problem $(P_1)$ given by \eqref{gop}  and optimal control $g_{\alpha_{op}}$ given by \eqref{galphaop} for $\alpha=15, 50, 100$. Notice  that as $\alpha$ increases, $g_{\alpha_{op}}$ decreases  to the limit $g_{op}$. 
In addition, we set different values of $n$ between $n=10$ and $n=100$. Recalling that $h=\tfrac{x_0}{n}=\tfrac{1}{n}$, for each $h$, we obtained the optimal discrete control ${g^h_{op}}$ to problem ${(P^h_1)}$ defined by \eqref{lema:goph} and the optimal discrete control  ${g^h_{\alpha_{op}}}$ to problem ${(P^h_{1\alpha})}$ given by \eqref{galphaop} for $\alpha=15, 50, 100$. For each $\alpha$ fixed, we observe the discrete solution ${g^h_{\alpha_{op}}}\to g_{\alpha_{op}}$ when $h\to 0$, i.e., $n\to\infty$.

\begin{figure}[h!]
\centering
\includegraphics[scale=0.5]{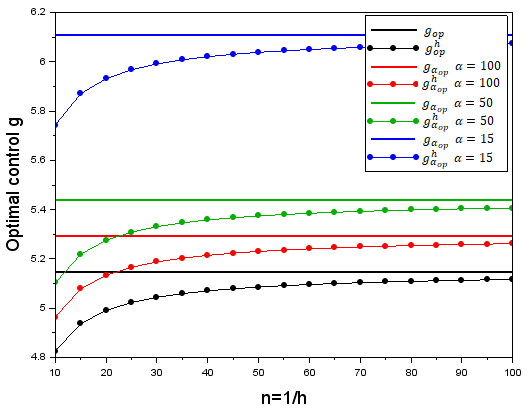}
\caption{Plot of $g_{op}$, $g^h_{op}$, $g_{\alpha_{op}}$ and $g^h_{\alpha_{op}}$ against $n=1/h$.}
\label{figgop}
\end{figure}

 \subsection{Control Variable $q$}

   In this subsection we ran some computational examples for the optimal boundary control problems $(P_2)$, ${(P^h_2)}$, $(P_{2\alpha})$ and ${(P^h_{2\alpha})}$. For each plot, we set $g=10, b=50, z_d=40$ and $M_2=1$.
  
    In Figure \ref{figJ2J2h} we plotted the continuous quadratic cost function $J_2$ given by \eqref{J2}  and the discrete cost function ${J^h_2}$ obtained in \eqref{J2h(q)} against $q$ for  $h=1/10$, $1/25$ and $1/50$. Observe that as $h$ decreases, function ${J^h_2}={J^h_2}(q)$  also decreases to the limit function $J_2=J_2(q)$. In a similar way, in Figure \ref{FigJ2alphaJ2halpha}, for $\alpha=100$, we obtained the continuous function $J_{2\alpha}$ and the discrete functions ${ J^h_{2 \alpha}}$ for $h=1/10,1/25$ and $1/50$. The convergences ${J^h_2}\to J_2$ and ${ J^h_{2 \alpha}}\to J_{2\alpha}$ when $h\to 0$ are in agreement with Lemmas \ref{Lema J2 J2h} and \ref{Lema 4.4}, respectively.
    
    \begin{figure}[h!]
\centering
\includegraphics[scale=0.5]{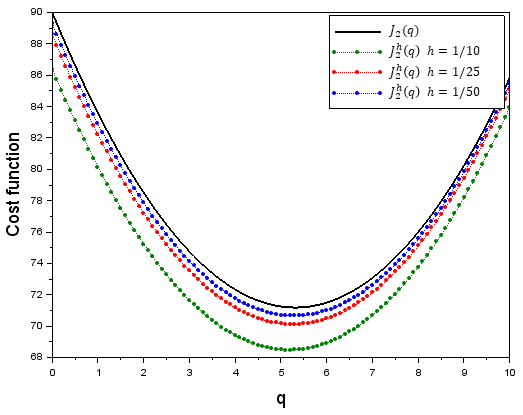}
\caption{Plot of $J_2$ and $J^h_2$ against $q$.}
\label{figJ2J2h}
\end{figure}

     \begin{figure}[h!]
     \centering
\includegraphics[scale=0.5]{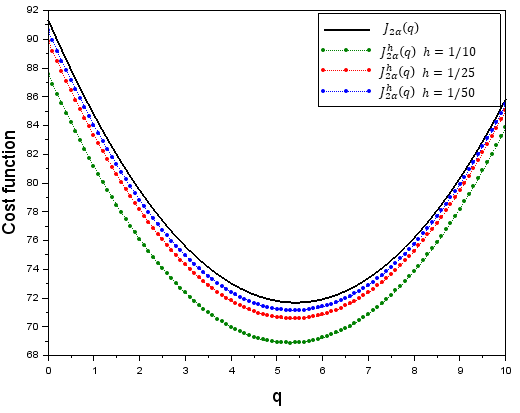}
\caption{Plot 
 of $J_{2\alpha}$ and {$J^h_{2\alpha}$} for $\alpha=100$ against $q$.}
\label{FigJ2alphaJ2halpha}

\end{figure}

     Moreover, Figure \ref{FigJ2dobleconvergencia}  shows the double convergence of ${ J^h_{2 \alpha}}\to J_2$ when $(h,\alpha)\to (0,\infty)$. We illustrate how ${ J^h_{2 \alpha}}$ gets closer to $J_2$ as the value of $h$ decreases and the value of $\alpha$ increases.

     \begin{figure}[h!]
     \centering
\includegraphics[scale=0.5]{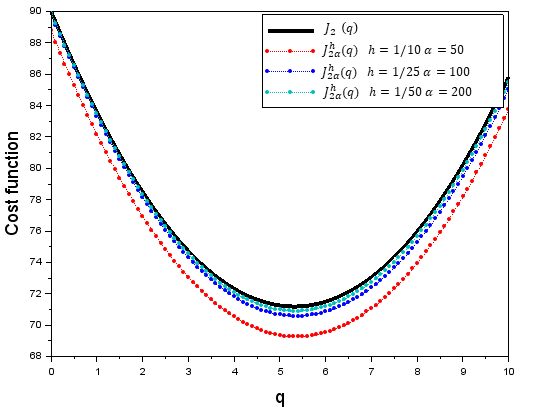}
\caption{Plot 
 of $J_2$ and {$J^h_{2\alpha}$ against $q$}.}
\label{FigJ2dobleconvergencia}

\end{figure}

In Figure \ref{Figqop} we plotted the continuous optimal control $q_{op}$  for problem $(P_2)$ given by \eqref{ec 4.2} and optimal control $q_{\alpha_{op}}$ given by \eqref{qalphaop} for $\alpha=50,  100, 200$. Notice  that as $\alpha$ increases, $q_{\alpha_{op}}$ decreases  to the limit $q_{op}$. 
In addition, we set different values of $n$ between $n=10$ and $n=100$. Recalling that $h=\tfrac{x_0}{n}=\tfrac{1}{n}$, for each $h$, we obtained the optimal discrete control   ${q^h_{op}}$ to problem ${(P^h_2)}$ defined by \eqref{qhop} and the optimal discrete control  ${ q^h_{\alpha_{op}}}$ to problem ${(P^h_{2\alpha})}$ given by \eqref{qhalphaop} for $\alpha=50, 100, 200$. For each $\alpha$ fixed, we observe the discrete solution ${ q^h_{\alpha_{op}}}\to q_{\alpha_{op}}$ when $h\to 0$, i.e., $n\to\infty$.

     \begin{figure}[h!]
          \centering
\includegraphics[scale=0.5]{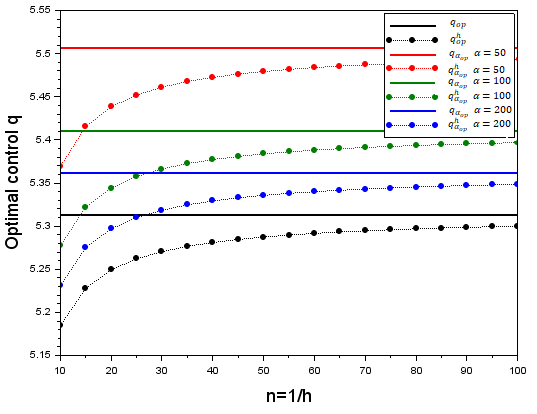}
\caption{Plot 
 of $q_{op}$, {$q^h_{op}$}, $q_{\alpha_{op}}$ and {$q^h_{\alpha_{op}}$ against $n=1/h$}.}
\label{Figqop}

\end{figure}


     \newpage
 
 \subsection{Control Variable $b$}
 
   In this section we obtain some computational examples for the optimal distributed control problems $(P_3)$, ${(P^h_{3})}$, $(P_{3\alpha})$ and ${(P^h_{3\alpha})}$. For each plot, we set $q=12, g=10, z_d=40$ and $M_3=1$.

\begin{figure}[H]
     \centering
\includegraphics[scale=0.5]{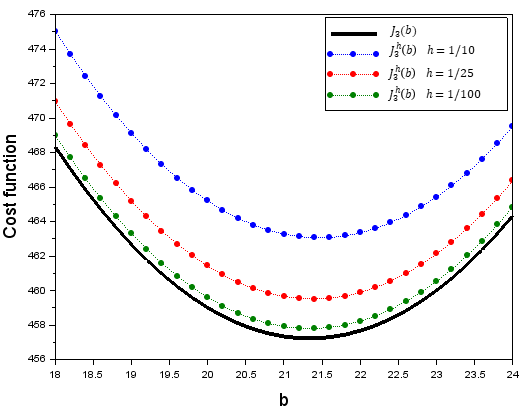}
\caption{Plot of $J_3$ and {$J^h_{3}$ against $b$}.}
\label{FigJ3J3h}
\end{figure}

  In Figure \ref{FigJ3J3h} we plotted the continuous quadratic cost function $J_3$ given by \eqref{J3-explicito} and the discrete cost function ${J^h_{3}}$ obtained in \eqref{J3h-explicito} against $g$ for  $h=1/10$, $1/25$ and $1/100$. Notice that as $h$ decreases, function ${J^h_{3}}={J^h_{3}}(b)$  also decreases to the limit function $J_3=J_3(b)$ in agreement with Lemma \ref{Lema J3 J3h}. In a similar manner, in Figure \ref{FigJ3alphaJ3halpha}, for $\alpha=50$, we obtained the continuous function $J_{3\alpha}$ and the discrete functions ${J^h_{3\alpha}}$ for $h=1/10,1/25$ and $1/100$. Observe the convergence of ${J^h_{3\alpha}}\to J_{3\alpha}$ as $h\to 0$.
     Moreover, Figure \ref{FigJ3dobleconvergencia} shows the double convergence of ${J^h_{3\alpha}}\to J_3$ when $(h,\alpha)\to (0,\infty)$. We illustrate how ${J^h_{3\alpha}}$ gets closer to $J_3$ as the value of $h$ decreases and the value of $\alpha$ increases. 

In Figure \ref{Figbop} we plotted the continuous optimal control $b_{op}$  for problem $(P_3)$ given by \eqref{bop-ubop}  and optimal control $b_{\alpha_{op}}$ given by \eqref{balphaop} for $\alpha=15, 50, 100$. Notice  that as $\alpha$ increases, $b_{\alpha_{op}}$ decreases  to the limit $b_{op}$. 
In addition, we set different values of $n$ between $n=10$ and $n=100$. Recalling that $h=\tfrac{x_0}{n}=\tfrac{1}{n}$, for each $h$, we obtained the optimal discrete control   ${b^h_{op}}$ to problem ${(P^h_{3})}$ defined by \eqref{bhop} and the optimal discrete control  ${b^h_{\alpha_{op}}}$ to problem ${(P^h_{3\alpha})}$ given by \eqref{bhalphaop} for $\alpha=15, 50, 100$. For each $\alpha$ fixed, we observe the discrete solution ${b^h_{\alpha_{op}}}$  decreases to $b_{\alpha_{op}}$ when $h\to 0$.

\begin{figure}[h!]
     \centering
\includegraphics[scale=0.5]{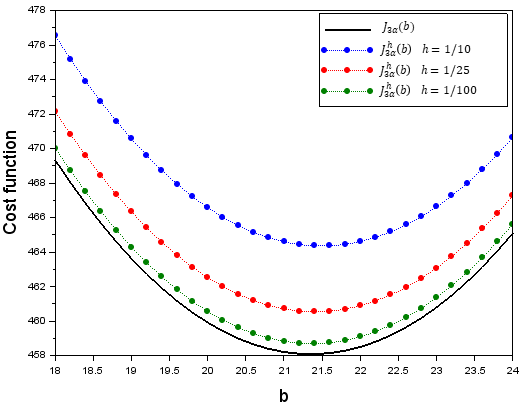}
\caption{Plot of $J_{3\alpha}$ and {$J^h_{3\alpha}$} for $\alpha=100$ {against $b$}.}
\label{FigJ3alphaJ3halpha}
\end{figure}

\begin{figure}[h!]
     \centering
\includegraphics[scale=0.5]{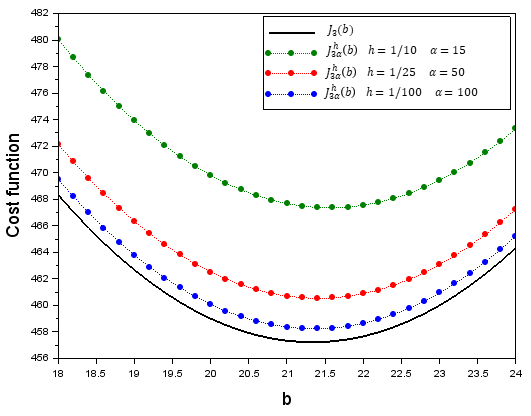}
\caption{Plot of $J_3$ and $J_{3\alpha}$ {against $b$}.}
\label{FigJ3dobleconvergencia}
\end{figure}

\newpage

\begin{figure}[H]
     \centering
\includegraphics[scale=0.5]{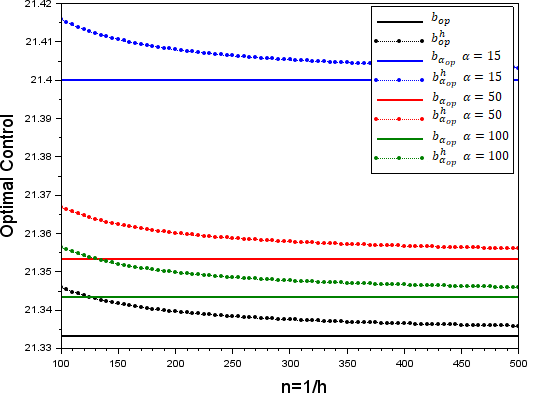}
\caption{Plot 
 of $b_{op}$, {$b^h_{op}$}, $b_{\alpha_{op}}$ and {$b^h_{\alpha_{op}}$ against $n=1/h$}.}
\label{Figbop}
\end{figure}


\section{Improvement of the Order of Convergence}\label{sec7}
In this section, we introduce alternative discrete solutions ${\widetilde{u}^h}$ and ${\widetilde{u}^h_\alpha}$ associated with systems $(S)$ and $(S_\alpha)$, respectively, and analyze the order of convergence of ${\widetilde{u}^h}$ to $u$ and of ${\widetilde{u}^h_\alpha}$ to $u_\alpha$ as $h \to 0^+$. The Neumann boundary condition on $\Gamma_2$ is approximated by a three-point backward finite-difference scheme. Moreover, for the discrete solution ${\widetilde{u}^h_\alpha}$, the Robin boundary condition on $\Gamma_1$ is approximated by a three-point forward finite-difference scheme. These higher-order boundary approximations lead to an improved order of accuracy.

We consider the system $(S)$ defined by Equations \eqref{ec 1.1} and \eqref{ec 1.2}. From this system, we define the discrete problem ${(\widetilde{S}^h)}$, where {for a fixed $h>0$}, $ {\widetilde{u}^h_i}$ approximates $u(x_i,y)$, for $i=1,\cdots, n+1$. Notice that from the Dirichlet condition on $\Gamma_1$, it follows immediately that ${\widetilde{u}^h_1} = b$.

For the interior nodes, we employ the classical centered second-order finite-difference approximation given in \eqref{aprox-deriv}, which leads to the discrete system \eqref{aprox-nodo-interior} for ${\widetilde{u}^h_i}$, $i=2,...,n$. 

For the Neumann boundary condition on $\Gamma_2$, we use the three-point backward \mbox{approximation}
\begin{equation}\label{aprox-Neumann-nueva}
\pder[u]{x}(x_{n+1},y)\approx \frac{3u(x_{n+1},y)-4u(x_n,y)+u(x_{n-1},y)}{2h}.
\end{equation}
Thus, the discrete Neumann condition can be written as
\begin{equation}
-2 q h = {
3\widetilde{u}^{h}_{n+1}
-4\widetilde{u}^{h}_{n}
+\widetilde{u}^{h}_{n-1}}.
\end{equation}
In addition, from \eqref{aprox-nodo-interior} for $i=n$, we obtain
\begin{equation}
-gh^2 = {\widetilde{u}^h_{n+1}-2\widetilde{u}^h_n+\widetilde{u}^h_{n-1}.}
\end{equation}
Subtracting the two previous equations, it follows that
\begin{equation}\label{aprox-nueva}
{-\widetilde{u}^h_n+\widetilde{u}^h_{n+1}}=\frac{gh^2}{2}-qh.
\end{equation}
Therefore, the system given by \eqref{aprox-nodo-interior} together with \eqref{aprox-nueva} can be written  as
\begin{equation}\label{sistema-nuevo}
{A w^h=\widetilde{T}^h}
\end{equation}
where $w^h = ({\widetilde{u}^h_i})_{i=2,\ldots,n+1}\in\mathbb{R}^{n}$ is the vector of unknowns, $A$ is the matrix given by \eqref{matrizA} and   ${\widetilde{T}^h}\in \mathbb{R}^n$ is the vector of independent terms:
    \begin{equation}\label{vectorT*}
     {\widetilde{T}^h}= \Big(-g\,h^2 - b, -g h^2, \ldots, -g h^2, -q\,h +\tfrac{gh^2}{2} \Big)^t.
\end{equation}
Notice that system \eqref{sistema-nuevo} differs from \eqref{sistema-tradicional} in the last component of the vector of independent terms.
Solving the linear system gives
\begin{equation}
{\widetilde{u}^h_i}=b+(i-1)h(gx_0-q)-\frac{gh^2}{2}(i-1)^2.
\end{equation}
Taking into account that for $i=1,\cdots, n$
\begin{equation}
\widetilde{m}_i={\frac{\widetilde{u}^h_{i+1}-\widetilde{u}^h_{i}}{x_{i+1}-x_i}}=gx_0-q-(2i-1)\frac{gh}{2},
\end{equation}
and
\begin{equation}
\widetilde{h}_i= {\widetilde{u}^h_i-\widetilde{m}_i x_i}= b+\frac{gh^2}{2} i (i-1),
\end{equation}
the linear approximation is given by ${\widetilde{u}^h}(x,y)=\widetilde{m}_i x+\widetilde{h}_i$, i.e.,

\vspace{-12pt}

\centering 
\begin{equation}\label{uhtilde}
{\widetilde{u}^h}(x,y)=\left(gx_0-q-(2i-1)\frac{gh}{2} \right)x+\frac{gh^2}{2} i (i-1)+b,\qquad x\in [x_i,x_{i+1}], \quad i=1,\cdots,n
\end{equation}

In the following lemma, we give some bounds for the approximate  function ${\widetilde{u}^h}:$

  \begin{Lemma} \label{Cota u uh tilde} The following  bounds hold:
    		\[\lVert u-{\widetilde{u}^h} \rVert _H \leq  D_1 h^2, \qquad \text{and}\qquad
\lVert \pder[u]{x}-\pder[{\widetilde{u}^h}]{x} \rVert _H \leq \widetilde{D}_{1}\, h, \]
    		where $D_1=\sqrt{\frac{x_0 y_0 }{120}} g $ 
   and $\widetilde{D}_1=\sqrt{\frac{x_0 y_0}{12}}\; g$.
    
    \end{Lemma}
\begin{proof} 
From the definition of the norm in space $H$ and using the expressions \eqref{Def: u y ualpha} and \eqref{uhtilde} for functions $u$ and ${\widetilde{u}^h}$, respectively, it follows that
\begin{equation}\label{cuenta-intermedia}
\begin{aligned}
\|u-{\widetilde{u}^h}\|_{H}^2
&= \int_0^{y_0} \int_0^{x_0} \big(u(x,y)-{\widetilde{u}^h}(x,y)\big)^2 \, \dee x \, \dee y \\[1ex]
&= y_0 \sum_{i=1}^n \int_{x_i}^{x_{i+1}} E_i^2(x)\, \dee x,
\end{aligned}
\end{equation}
where 
\[
E_i(x) = u(x,y) - {\widetilde{u}^h}(x,y), \quad x \in [x_i, x_{i+1}], \; y \in [0, y_0].
\] 
Note that, within each subinterval, $E_i(x)$ depends only on $x$ and the index $i$, but not on $y$, since both $u$ and ${\widetilde{u}^h}$ are constant along the $y$-direction.  

A direct computation yields
\begin{equation}\label{Ei}
E_i(x) = \frac{g}{2} \Big(-x^2 + (2i-1) h x - h^2 i(i-1) \Big) 
= -\frac{g}{2} (x - i h)(x - (i-1) h).
\end{equation}

Then,
\begin{equation}\label{integralEi}
\begin{aligned}
\int_{x_i}^{x_{i+1}} E_i^2(x)\,\dee x
&= \frac{g^2}{4}
\left[
\frac{(x-i h)^5}{5}
+ \frac{h}{2}(x-i h)^4
+ \frac{h^2}{3}(x-i h)^3
\right]_{x_i}^{x_{i+1}} \\[0.8ex]
&= \frac{g^2}{4}
\left(
\frac{h^5}{5}
- \frac{h^5}{2}
+ \frac{h^5}{3}
\right)
= \frac{g^2}{120}\,h^5 .
\end{aligned}
\end{equation}
As a consequence, from \eqref{cuenta-intermedia}, it follows that
$$\|u-{\widetilde{u}^h}\|_{H}^2=y_0 \frac{g^2 h^5}{120} n= \frac{x_0 y_0 g^2}{120} h^4,$$
and then
$$\|u-{\widetilde{u}^h}\|_{H}=\sqrt{\frac{x_0 y_0 }{120}} g h^2. $$
In addition,
\begin{equation}\label{norma-derivada-cuadrado}
\begin{aligned}
\big\|\pder[u]{x}-\pder[{\widetilde{u}^h}]{x}\big\|_{H}^{2}
&= \int_0^{y_0}\int_0^{x_0}
\big(\pder[u]{x}(x,y)-\pder[{\widetilde{u}^h}]{x}(x,y)\big)^2
\,\dee x\,\dee y \\[1ex]
&= y_0 \sum_{i=1}^n \int_{x_i}^{x_{i+1}} F_i^2(x)\,\dee x ,
\end{aligned}
\end{equation}
where
$$
F_i(x)
= \pder[u]{x}(x,y)-\pder[{\widetilde{u}^h}]{x}(x,y)
= -g\!\left(x-\frac{(2i-1)h}{2}\right),
$$
 for $x\in[x_i,x_{i+1}]$.
 Then,
 \[
\begin{aligned}
\int_{x_i}^{x_{i+1}} F_i^2(x)\,\dee x
&= \int_{x_i}^{x_{i+1}}
g^2\!\left(x-\frac{(2i-1)h}{2}\right)^2 \dee x \\[0.8ex]
&= g^2
\left[
\frac{1}{3}\left(x-\frac{(2i-1)h}{2}\right)^3
\right]_{x_i}^{x_{i+1}} \\
&= g^2\,
\frac{1}{3}\left(\frac{h^3}{8}+\frac{h^3}{8}\right)
= \frac{g^2}{12}\,h^3 .
\end{aligned}
\]
Therefore, from \eqref{norma-derivada-cuadrado}, we have
\[
\left\|\pder[u]{x}-\pder[{\widetilde{u}^h}]{x}\right\|_{H}^{2}
= y_0\,\frac{g^2}{12}\,n h^3
= \frac{x_0 y_0 g^2}{12}\,h^2 ,
\]
and finally
\[
\left\|\pder[u]{x}-\pder[{\widetilde{u}^h}]{x}\right\|_{H}
= \sqrt{\frac{x_0 y_0}{12}}\; g\, h .
\]
\end{proof}

 \begin{Remark}
We emphasize that by improving the approximation of the Neumann boundary condition on $\Gamma_2$, the convergence order of the error $\|u-{\widetilde{u}^h}\|_H$ is increased to second order, namely, $O(h^2)$. {The improvement is entirely due to the modification in the last component of vectors $T^h$ and $T^h_\alpha$  in systems $(S)$ and $(S_\alpha)$, respectively, where a term of order $h^2$ appears.}
This enhancement leads to a more accurate numerical approximation while remaining fully consistent with the theoretical convergence results established in~\cite{CaMa2002,Hi2005}.
\end{Remark}

\begin{Remark}
The linear system \eqref{sistema-nuevo} obtained by using the three-point backward finite-difference approximation for the Neumann boundary condition on $\Gamma_2$ can be equivalently interpreted by introducing a ghost point $x_{n+2}$ outside the computational domain and assuming that the discrete differential equation holds at the boundary node $x_{n+1}$.
Indeed, assuming that the equation is satisfied at ${\widetilde{u}^h_{n+1}}$, we have
\[
- g h^2 = {\widetilde{u}^h_{n+2}
- 2\widetilde{u}^h_{n+1}
+ \widetilde{u}^h_{n},}
\]
while the Neumann boundary condition is approximated by
\[
{\frac{\widetilde{u}^h_{n+2}-\widetilde{u}^h_{n}}{2h} }= -q .
\]
Eliminating the ghost value $\widetilde{u}_{n+2}$ from these two expressions yields
\[
{-\widetilde{u}^h_{n}+\widetilde{u}^h_{n+1}}
= -q h + \frac{g h^2}{2},
\]
which coincides with the boundary equation obtained in \eqref{aprox-nueva}.
Hence, the three-point backward finite-difference approximation of the Neumann condition is  consistent with the ghost-point formulation and leads to the same discrete system.
\end{Remark}

\begin{flushleft}

Analogously to the analysis of system $(S)$, we propose a new discrete approximation ${\widetilde{u}^h_\alpha}$ for system $(S_\alpha)$ and study the order of convergence of ${\widetilde{u}^h_\alpha}$ to $u_\alpha$ as $h \to 0^+$. The associated discrete system $(\widetilde{S}_{h\alpha})$ employs a three-point backward finite-difference approximation for the Neumann boundary condition on $\Gamma_2$ and a three-point forward finite-difference approximation for the Robin boundary condition on $\Gamma_1$, leading to improved accuracy.

We consider system $(S_\alpha)$ defined by Equations \eqref{ec 1.1} and \eqref{ec 1.3} and define ${\widetilde{u}^h_{\alpha,i}} \approx u_\alpha(x_i, y)$. 

For the interior nodes, $i=2,\dots,n$, we employ the classical centered second-order finite-difference approximation given in \eqref{aprox-deriv}:
\begin{equation}\label{nodointerior-ualpha}
{\widetilde{u}^h_{\alpha,i+1}-2\widetilde{u}^h_{\alpha,i}+\widetilde{u}^h_{\alpha,i-1}}=-g h^2.
\end{equation}
For the Robin boundary at $\Gamma_1$, we use the three-point forward approximation:
\begin{equation}
{\frac{-3 \widetilde{u}^h_{\alpha,1} + 4 \widetilde{u}^h_{\alpha,2} - \widetilde{u}^h_{\alpha,3}}{2h}} = \alpha (\widetilde{u}_{\alpha,1}-b).
\end{equation} 
Combining this expression with the interior equation at $i=2$ yields the simplified \mbox{discrete~condition}
\begin{equation}\label{robin-modificada}
{-(1+\alpha h)\widetilde{u}^h_{\alpha,1} + \widetilde{u}^h_{\alpha,2}} =- \alpha h b -\frac{g h^2}{2}.
\end{equation}
For the Neumann boundary at $\Gamma_2$ we use the three-point backward approximation:
\begin{equation}
{3\widetilde{u}^h_{\alpha,n+1}-4\widetilde{u}^h_{\alpha,n}+\widetilde{u}^h_{\alpha,n-1}}=-2 q h.
\end{equation}
Combining with the interior equation for $i=n$ gives
\begin{equation} \label{neumann-modificada}
{-\widetilde{u}^h_{\alpha,n}+\widetilde{u}^h_{\alpha,n+1}}=-q h + \frac{g h^2}{2}.
\end{equation}
The system given by \eqref{nodointerior-ualpha}, \eqref{robin-modificada} and \eqref{neumann-modificada} can be rewritten as 
\begin{equation}\label{sistemaalpha-nuevo}
{ A_\alpha w^h_\alpha=\widetilde{T}^h_\alpha}
\end{equation}
where ${ w^h_\alpha} = ({\widetilde{u}^h_{\alpha, i}})_{i=1,\ldots,n+1}\in\mathbb{R}^{n+1}$ is the vector of unknowns, $A_\alpha$ is the matrix given by \eqref{matrizAalpha} and   ${\widetilde{T}^h_\alpha}\in \mathbb{R}^{n+1}$ is  the vector of independent terms:
   \begin{equation}\label{Talpha*}
 {\widetilde{T}^h_\alpha}= \Big(-\alpha\,b\,h-\tfrac{gh^2}{2}, -g h^2, \ldots, -g h^2, -q\,h+\tfrac{gh^2}{2}  \Big)^t \in  \mathbb{R}^{n+1}.
  \end{equation}
It should be noted that only the first and last components of \({\widetilde{T}^h_\alpha}\)  differ from those in \(T_\alpha\) given by \eqref{Talpha}.

The solution of system \eqref{sistemaalpha-nuevo} is given by
\begin{equation}
\label{eq:solucion_indices_alpha}
{\widetilde{u}^h_{\alpha,i}} =
b + \frac{1}{\alpha} \Big( g x_0 - q \Big)
+ (i-1) h \,(g x_0 - q) - \frac{g}{2} \big((i-1) h\big)^2,
\quad i=1,\dots,n+1.
\end{equation}
We define the linear interpolation on each subinterval $[x_i, x_{i+1}]$ by
\begin{equation}\label{uhalpha-nueva}
{\widetilde{u}^h_\alpha}(x,y) = \widetilde{m}_{\alpha,i} x + \widetilde{h}_{\alpha,i}, \quad x\in[x_i,x_{i+1}],\; y\in[0,y_0], 
\end{equation}
where
\begin{align}
\widetilde{m}_{\alpha,i} &= g x_0 - q - gh \left(i - \frac{1}{2}\right), \quad i=1,\dots,n,\\[2mm]
\widetilde{h}_{\alpha,i} &= b + \frac{1}{\alpha} (g x_0 - q ) + \frac{g h^2}{2} (i-1) i, \quad i=1,\dots,n.
\end{align}
From the previous expressions, we derive the following lemma.
\begin{Lemma} \label{Cota u uh alpha} The following bounds hold:
\[
\lVert u_\alpha - {\widetilde{u}^h_\alpha} \rVert_H \leq D_2 h^2, \qquad
\lVert \pder[u_\alpha]{x} - \pder[{\widetilde{u}^h_\alpha}]{x} \rVert_H \leq \widetilde{D}_2 h,
\]
where $D_2 = \sqrt{\frac{x_0 y_0}{120}} g$ and $\widetilde{D}_2 = \sqrt{\frac{x_0 y_0}{12}} g$.

\end{Lemma}

\begin{proof}
By the definition of the $H$-norm, and using the expression for $u_\alpha$ in \eqref{Def: u y ualpha} as well as the definition of ${\widetilde{u}^h_\alpha}$ in \eqref{uhalpha-nueva}, it follows that
\begin{equation}
\lVert u_\alpha - {\widetilde{u}^h_\alpha} \rVert_H^2 
= y_0 \sum_{i=1}^n \int_{x_i}^{x_{i+1}} E_{\alpha,i}^2(x)\, dx,
\end{equation}
where 
\[
E_{\alpha,i}(x) = u_\alpha(x,y) - {\widetilde{u}^h_\alpha}(x,y) \] \begin{equation}
\label{eq:error_expandido}
E_{\alpha,i}(x)= -\frac{g}{2} x^2 +\frac{ g}{2} \left(2i -1\right) h x  - \frac{g h^2}{2} (i^2 - i), \quad x \in [x_i, x_{i+1}]
\end{equation}
We can notice that $E_{\alpha,i}(x)=E_i(x)$ where $E_i(x)$ is given by \eqref{Ei}. Therefore, from \eqref{integralEi}, it follows immediately that 
\[
\lVert u_\alpha - {\widetilde{u}^h_\alpha} \rVert_H^2 \leq \frac{x_0y_0g^2}{120} h^4. 
\]
In addition,
\begin{equation}\label{norma-derivadaUALPHA-cuadrado}
\begin{aligned}
\bigl\|\pder[u]{x}-\pder[{\widetilde{u}^h}]{x}\bigr\|_{H}^{2}
&= \int_{0}^{y_0}\!\int_{0}^{x_0}
\bigl(\pder[u]{x}(x,y)-\pder[{\widetilde{u}^h}]{x}(x,y)\bigr)^2
\,\dee x\,\dee y \\[1ex]
&= y_0 \sum_{i=1}^{n} \int_{x_i}^{x_{i+1}}
g^2\!\left(x-\frac{(2i-1)h}{2}\right)^2 \,\dee x \\[1ex]
&= y_0\,\frac{g^2}{12}\,h^3\,n
= \frac{x_0 y_0 g^2}{12}\,h^2 .
\end{aligned}
\end{equation}
\end{proof}

  \section{Conclusions}\label{sec8}

Applying the finite difference method, we derived discrete systems $({S^h})$ and ${(S^h_\alpha)}$ and discrete optimization problems ${(P^h_{i})}$ and ${(P^h_{i\alpha})}$, $i=1,2,3$, where $\alpha>0$ is a parameter that represents the heat transfer coefficient on a portion of the boundary of the domain. Explicit discrete solutions were obtained, and convergence results as discretization step $h \to 0$ and parameter $\alpha \to \infty$ were proved. Error estimations were also obtained as a function of step $h$. Some numerical computations were  provided in order to illustrate the  theoretical results.

The obtained results showed that the proposed numerical approach provided first-order accurate approximations for both state systems $(S)$ and $(S_\alpha)$ and associated optimal control problems $(P_i)$ and $(P_{i\alpha})$, $i=1,2,3$, and that the discrete solutions converged to the corresponding continuous ones as discretization step $h \to 0$.

Finally, for systems $(S)$ and $(S_\alpha)$, an alternative discretization of the Neumann boundary condition on $\Gamma_2$ and of the Robin boundary condition on $\Gamma_1$ for $(S_\alpha)$ was considered. By modifying the approximation of these boundary conditions, the order of convergence of the numerical solution was improved, leading to a more accurate approximation.

A main limitation of the present work is that the analysis is restricted to rectangular domains, which allows the derivation of explicit solutions and simplifies the numerical implementation. As a future development, the proposed methodology is expected to be extended to more general domains, including polar and spherical coordinate systems.


%
%

\section*{Acknowledgments}
The authors would like to thank the support of project O06-24CI1901 from Universidad Austral, Rosario, Argentina, and project PIP Nº 11220220100532 from CONICET.

\end{flushleft}

\end{document}